\documentclass[11pt]{amsart}
\usepackage{amssymb, graphicx, color}
\usepackage{amsfonts}
\usepackage{amsthm}
\usepackage{enumerate}
\usepackage[mathscr]{eucal}
\usepackage{graphicx}

\usepackage[pdftex, bookmarksnumbered, bookmarksopen, colorlinks, citecolor=blue, linkcolor=blue]{hyperref}

\allowdisplaybreaks

\newtheorem{thm}{Theorem}[section]
\newtheorem{defi}[thm]{Definition}

\newtheorem{lem}[thm]{Lemma}

\newtheorem{prop}[thm]{Proposition}

\newtheorem{rk}[thm]{Remark}

\theoremstyle{definition}
\theoremstyle{remark}
\numberwithin{equation}{section}

\newcommand{\R}{{\mathbb R}}

\newcommand{\Z}{{\mathbb Z}}
\newcommand{\C}{{\mathbb C}}

\newcommand{\M}{{\mathcal M}}

\newcommand{\n}{{\mathcal N}}
\newcommand{\bs}{\begin{split}}
\newcommand{\es}{\end{split}}

\theoremstyle{plain}

\newcommand{\be}{\begin{eqnarray*}}
\newcommand{\ee}{\end{eqnarray*}}
\newcommand{\beq}{\begin{align}}
\newcommand{\eeq}{\end{align}}

\def\1{\mathbf{1}}

\begin{document}

%
%
%
%
%
%
%
%
\setcounter{page}{1}
\title[Hermite expansions]
{Noncommutative analysis of Hermite expansions}



\author[B. Xu]{Bang Xu}
\address{Department of Mathematics\\
	University of Houston\\
	Houston TX 77204\\
	USA;
Department of Mathematical Sciences\\
Seoul National University\\
Seoul 08826\\
Republic of Korea}

\email{bangxu@whu.edu.cn}

\thanks{This work was supported by National Natural Science Foundation of China (No. 12071355), the Basic Science Research Program
	through the National Research Foundation of Korea (NRF) Grant: NRF-2022R1A2C1092320
	and the Samsung Science and Technology Foundation under Project Number: SSTF-BA2002-01.}


\date{\today}
\begin{abstract}
This paper is devoted to the study of semi-commutative harmonic analysis associated to Hermite semigroups. In the first part,
we establish the noncommutative maximal inequalities for Bochner-Riesz means associated with Hermite operators and then obtain the corresponding pointwise convergence theorems. In particular, we develop a noncommutative Stein\textquoteright s theorem of Bochner-Riesz means for Hermite operators. The second part of this paper deals with two multiplier theorems for Hermite operators. Our analysis on this part is based on a noncommutative analogue of the classical Littlewood-Paley-Stein theory associated with Hermite semigroup.
\end{abstract}

\maketitle

\section{Introduction}
The theme of this paper follows the current research direction of noncommutative harmonic analysis. Motivated by operator space theory, Pisier and Xu \cite{PX1} developed a pioneering work on noncommutative martingale theory; since then, many classical theories have been successfully transferred to the noncommutative or quantum setting. For instance, it is inspired by noncommutative Burkholder-Gundy inequalities that Junge, Le Merdy and Xu \cite{JMX} initiated the research of noncommutative harmonic analysis; they established the noncommutative version of Littlewood-Paley-Stein theory by using the relationship between completely bounded $H^\infty$ functional calculus and quantum Markov semigroups. Later on, Junge and Mei \cite{JM2} introduced the Hardy and BMO spaces associated to quantum Markov semigroups. Again by quantum Markov semigroups, Junge et al \cite{JM1,JMP14,JMP18} studied H\"ormander-Mihlin
Fourier multipliers and Riesz transforms on group von Neumann algrbras (see \cite{JRD22,MRX22} for more results); Chen et al \cite{CXY13} developed systematically harmonic analysis on quantum tori (see also \cite{GJP17} regarding singular integral theory on quantum Euclidean spaces). We remark that via the transference technique  (see e.g. \cite{JMX,JM1,JM2}), harmonic analysis associated to quantum semigroups and semi-commutative harmonic analysis play an important role in
these fundamental works.


Semi-commutative harmonic analysis seems to be the easiest one in noncommutative theory, but it often requires new ideas and insights. The first notable work in this direction is due to  Mei \cite{M}; the  author gave a systematic study on the operator-valued Hardy spaces and BMO spaces, which incidentally solved an open question in matrix-valued harmonic analysis arising from prediction theory. More precisely, based on the noncommutative Doob's maximal inequality \cite{J1} or Cuculescu's maximal weak type $(1,1)$ result for martingales \cite{Cuc}, Mei established the operator-valued Hardy-Littlewood maximal inequalities (see \cite{JP1,HX} for more approaches). Furthermore, in the same paper, Mei proved that these operator-valued Hardy spaces, which are defined by the Littlewood-Paley $g$-function or Lusin function associated to Poisson kernel, are norm equivalent in the sense of noncommutative $L_{p}$-norms (see \cite{XXX16,HY2013} for a general characterization of operator-valued Hardy spaces).

We remark that most of works mentioned above, including Mei's Hardy/BMO spaces, Junge-Le Merdy-Xu's Littlewood-Paley-Stein theory and Hardy/BMO spaces over quantum tori etc, belong to harmonic analysis associated to quantum Markov semigroups. So it is natural to consider semigroups beyond Markov in noncommutative harmonic analysis. The typical example in classical case is the Hermite semigroup.

The purpose of this paper is to investigate Fourier analysis associated with the Hermit semigroup acting on operator-valued $L_p$ functions, which can be regarded as a case study of the general research program on noncommutative harmonic analysis, for instance the spectral multipliers and Littlewood-Paley-Stein theory associated to quantm semigroups beyond Markov.

We first set some notation. Let $H=-\Delta+|x|^{2}$ be the Hermite operator---the generator of Hermite semigroup---on $\R^d$. Recall that the Hermite functions $H_{m}(t)$ on $\R$ are defined by
$$H_{m}(t)=(-1)^{m}\exp(-t^{2})(\frac{d}{dt})^{m}\{\exp(-t^{2})\},\ m=0,1,2,....$$
The normalised Hermite functions, denote by  $\phi_{m}(t)$ (see (\ref{899})), form an orthonormal basis of
$L_{2}(\R)$.
For any multiindex $\nu$,  the $d$-dimensional Hermite functions $\{\Phi_{\nu}\}_{\nu}$ are given by the tensor product of one dimensional normalised Hermite functions (see (\ref{8919})), which form a complete orthonormal system in $L_{2}(\R^{d})$. Therefore, for any $f\in L_{2}(\R^{d})$, we have the  Hermite expansion 
$$f(x)= \sum_{\nu}\langle f,\Phi_{\nu}\rangle\Phi_{\nu}(x).$$

Research on Hermite expansions could be traced back to 1965s. Askey and Wainger \cite{AW} proved that the Hermite expansion converges if and only if $\frac{4}{3}<p<4$ for all $f\in L_{p}(\R^d)$. Therefore, it is necessary to find suitable summability methods since the expansion fails to converge for $p$ lying outside the interval $(\frac{4}{3},4)$. The most typical study object is the  Bochner-Riesz means. 
For $R>0$, the Bochner-Riesz means of order $\alpha$ is defined by  
$$S_{R}^{\alpha}f(x)=\sum_{k=0}^{\infty}\Big(1-\frac{2k+d}{R}\Big)_{+}^{\alpha}P_{k}f(x),$$
where $P_{k}$ denotes the Hermite projection operator (see (\ref{99232})).

The boundedness theory of Bochner-Riesz means associated to Hermite expansion was originally established by Thangavelu. In one dimensional case, it is known \cite{T1} that if $\alpha>1/6$, then for any $f\in L_{p}(\R)$ with $1\leq p<\infty$,
\begin{equation}\label{mean1}
\|S_{R}^{\alpha}f-f\|_{L_{p}(\R)}\rightarrow0\ \mbox{as}\ R\rightarrow\infty.
\end{equation}
In higher dimensions $(d\geq2)$, the $L_p$ convergence of $S_{R}^{\alpha}f$ is not so well understood yet. When $\alpha>(d-1)/2$, Thangavelu \cite{T2} showed that for any $f\in L_{p}(\R^d)$ with $1\leq p<\infty$,
\begin{equation}\label{mean2}
	\|S_{R}^{\alpha}f-f\|_{L_{p}(\R^d)}\rightarrow0\ \mbox{as}\ R\rightarrow\infty.
\end{equation}

Concerning almost everywhere convergence of $S_{R}^{\alpha}f$, it is known from \cite{T1,T2} that if $\alpha>1/6$, then for any $f\in L_{p}(\R)$ with $1\leq p<\infty$,
\begin{equation}\label{ae1}
	S_{R}^{\alpha}f\rightarrow f\ \mbox{a.e.}\ \mbox{as}\ R\rightarrow\infty.
\end{equation}
In higher dimensions, if $\alpha>(d-1)/2$, then for any $f\in L_{p}(\R^d)$ with $1\leq p<\infty$,
\begin{equation}\label{ae2}
	S_{R}^{\alpha}f\rightarrow f\ \mbox{a.e.}\ \mbox{as}\ R\rightarrow\infty.
\end{equation}

The work of Thangavelu has already found application in \cite{E1,EV}. For more results regarding the Bochner-Riesz means associated to Hermite operators, we refer the reader to \cite{CLSY20,CDHLY21} and  references therein.

In classical harmonic analysis, it is well-known that the convergence properties of Bochner-Riesz means associated to Fourier series are among the most important problems. Due to the noncommutativity, the study of noncommutative Bochner-Riesz means seems to be more challenging. For instance, Chen et al \cite{CXY13} had to found a much more technical proof when they dealt with the boundedness of the maximal Bochner-Riesz means on quantum tori. Later on, Lai \cite{L2022} obtained the full $L_p$-bounded of Bochner-Riesz means on two-dimensional quantum tori by establishing the sharper estimates of noncommutative Kakeya maximal functions and some geometric estimates in the plain.

The first part of this paper is to study the convergence properties of noncommutative Bochner-Riesz means associated to Hermite expansion. More precisely, we establish the noncommutative analogues of (\ref{mean1})-(\ref{ae2}). 
Let $\M$ be a von Neumann algebra equipped with a normal semifinite
faithful (abbrieviated as \emph{n.s.f.}) trace $\tau$ and $\mathcal N=L_{\infty}(\R^{d})\overline{\otimes}\M$ be a tensor von Neumann algebra equipped with a tensor trace $\varphi=\int\otimes \tau$. Let $L_p(\M)$ and $L_p(\mathcal N)$ be the noncommutative $L_p$-spaces associated to the pairs $(\M,\tau)$ and $(\mathcal N,\varphi)$, where $L_p(\mathcal N)$  can be identified as the space of $L_p(\mathcal M)$-valued $p$-th integrable functions on $\R^d$, that is $L_p(\mathcal N)=L_p(\R^{d};L_p(\mathcal M))$ whenever $0< p<\infty$. 

In Section 3, we consider the mean convergence of $S_{R}^{\alpha}f$ whenever $f\in L_p(\mathcal N)$ (see Theorem \ref{tc}), that is to establish the noncommutative analogues of (\ref{mean1}) and (\ref{mean2}), whose proof is based on the kernel estimates obtained by Thangavelu \cite{T1,T2}. In the latter part of Section 3, we present the maximal inequalities of the sequence of operators $(S_{R}^{\alpha})_{R>0}$ (see Theorem \ref{t6}, \ref{local} and \ref{t111}), which are crucial tools to obtain the noncommutative analogues of (\ref{ae1}) and (\ref{ae2}). However, compared with the $L_p$ convergence, the study of maximal inequalities is much more complicated. An immediate difficulty is that the
classical maximal function of the form $\sup_{i\in I}|f_i|$ no longer exists in general whenever $(f_i)_{i\in I}$ is a sequence of operators. Even though the definition of noncommutative weak type
$(1, 1)$ maximal inequalities exists in the early stage of noncommutative ergodic theory, the formulation of $L_p$ maximal inequality was not proposed until the $\ell_\infty$-valued noncommutative $L_p$-spaces  appeared (see Section 2)  introduced by  Pisier \cite{Pis01} and Junge \cite{J1} two decades ago
and has been widely used since then.

The second aspect of this paper is to consider certain multiplier transforms for Hermite expansion. Given a function $\mu$ on the set of positive integers, we define the operator $T_{\mu}$ by the prescription
$$T_{\mu}f(x)=\sum_{n=0}^{\infty}\mu(2n+d)P_{n}f(x).$$
To find a sufficient condition on the function $\mu$ such that $T_{\mu}$  is bounded on $L_p(\R^d)$ ($1<p<\infty$) is one of the mian task in Thangavelu's work \cite{T4}. Motivated by the classical Marcinkiewicz multiplier theorem for Fourier series, Thangavelu introduced the finite difference operators, which are defined inductively as follows:
\begin{equation*}
\delta\mu(N)=\mu(N+1)-\mu(N)
\end{equation*}
and for $n\geq1$, they are defined by
\begin{equation*}
\delta^{n+1}\mu(N)=\delta^{n}\mu(N+1)-\delta^{n}\mu(N).
\end{equation*}
Accordingly, Thangavelu \cite{T4} established the  Marcinkiewicz multiplier theorem for Hermite expansions. More precisely, if $\mu$ satisfies the condition $|\delta^{r}\mu(N)|\leq CN^{-r}\ \ \mbox{for}\ \ r=0,1,...,n$
with $n>d/2$, then for $1<p<\infty$,
\begin{equation}\label{mul3}
\|T_{\mu}f\|_{p}\leq C\|f\|_{p}\ \forall f\in L_{p}(\R^d).
\end{equation}

In Section 4, we extend (\ref{mul3}) to the operator-valued setting (see Theorem \ref{555}). The main ingredient is the noncommutative Littlewood-Paley-Stein theory for the $g$-functions associated with the Hermite semigroup $H^t$ (see (\ref{109})). Since the Hermite semigroup $H^{t}$ fails to satisfy the nice condition $H^{t}1=1$, the previous results on Markov semigroup can not be adapted here. Instead, in view of the explicit form of the associated kernel, we can make use of kernel estimates and operator-valued Calder{\'o}n-Zygmund  (abbrieviated as CZ) theory to overcome the difficulties. The semi-commutative  CZ singular integral theory was  established by Mei and Parcet et al. In particular,  Parcet \cite{JP1} 
made use of the tool of noncommutative martingales theory to formulate a  noncommutative version of CZ decomposition. Consequently, Parcet obtained all the $L_p$ estimates of standard CZ operators acting on operator-valued functions, which finds  an unexpected application in solving the Nazarov-Peller conjecture arising from the perturbation theory \cite{CPSZ}. Recently, Parcet's decomposition and CZ arguments are greatly improved in a later work \cite{CCP}, see also \cite{HLX}.

In \cite{T4}, Thangavelu also investigated a kind of oscillation operator related to Hermite expansion, where the function $\mu$ is given by 
$$\mu(n)=(2n+d)^{-\alpha}e^{(2n+d)it}.$$
This defines the operator $T_{t}^{\alpha}$ by
$$T_{t}^{\alpha}f(x)=\sum_{n=0}^{\infty}(2n+d)^{-\alpha}e^{(2n+d)it}P_{n}f(x).$$
It is known from \cite{T4} that if $\alpha=d|1/p-1/2|$, then for $1<p<\infty$,
\begin{equation}\label{mul4}
\|T_{t}^{\alpha}f\|_{L_p(\R^d)}\leq C\|f\|_{L_p(\R^d)}.
\end{equation}

In Section 5, we extend Thangavelu's work (\ref{mul4}), whenever $f$ is considered as an operator-valued function (see Theorem \ref{C3}). The main difficulty lie in the fact that the kernel associated to such operator has oscillating factor $e^{ix\cdot t}$, the study of this operator does not fall into the scope of the noncommutative CZ theory.
Fortunately, the associated kernel can be calculated explicitly (see Lemma \ref{2002}); and we may use an $H_1\rightarrow L_1$ endpoint estimate of $T_{t}^{\alpha}$ at critical point $\alpha=d/2$
by using the atom characterization of $H_1$ Hardy space introduced by Mei \cite{M}. Finally, together with Stein's analytic interpolation, we obtain the desired result.

\textbf{Notation:} Throughout the paper we write
$X\lesssim Y$ for nonnegative quantities $X$ and $Y$ to imply that there exists some inessential constant $C>0$ such that $X\le CY$ and we write $X\thickapprox Y$ to mean that $X\lesssim Y$ and $Y \lesssim X$.

\section{Preliminaries}
\subsection{Noncommutative $L_p$-spaces}
Throughout the paper, $\M$ always denote a semifinite von Neumann algebra equipped with a normal semifinite faithful trace $\tau$.
Let $\M_{+}$ be the positive part of $\M$ and $\mathcal{S}_{\M+}$ be the set of all $x\in\M_{+}$ such that $\tau (\mathrm{supp} (x) ) < \infty,$ where $\mathrm{supp} (x)$ means the support of $x$. Denote by $\mathcal{S}_{\M}$ the linear span of $\mathcal{S}_{\M+}$. Then $\mathcal{S}_{\M}$ is a $w^{*}$-dense $\ast$-subalgebra of $\M$. Given $1\leq p<\infty$ and $x\in\mathcal{S}_{\M}$, we set
$$\|x\|_{p}=\big(\tau(|x|^p)\big)^{1/p},$$
where $|x|=(x^{\ast}x)^{1/2}$ is the modulus of $x$. Then one can check that $\|\cdot\|_{p}$ is a norm on $\mathcal{S}_{\M}$. The completion of $(\mathcal{S}_{\M},\|\cdot\|_{p})$ is the so-called noncommutative $L_{p}$-space associated with $(\M,\tau)$, which is simply written as $L_{p}(\M)$. As usual, we set $L_{\infty}(\M) = \M$ equipped with the operator norm $\|\cdot\|_{\M}$. Denote by $L_{p}(\M)_{+}$ the positive part of $L_{p}(\M)$.

A closed and densely defined operator
$x$ affiliated with $\mathcal{M}$ is called \emph{$\tau$-measurable} if
there is $\lambda > 0$ such that $$\tau \big( \chi_{(\lambda,\infty)}
(|x|) \big) < \infty,$$
where $\chi_{\mathcal{I}}(x)$ is the spectral
decomposition of $x$ and $\mathcal{I}$ is a measurable subset of $\R$.
Let $L_0(\M)$ be the $*$-algebra of \emph{$\tau$-measurable}
operators. For $1\leq p<\infty$, the noncommutative weak $L_p$-space
$L_{p,\infty}(\mathcal{M})$ is defined as the set of all $x$ in $L_0(\M)$ for which the following  quasi-norm is finite
$$\|x\|_{p,\infty}=\sup_{\lambda > 0}\lambda\tau \big( \chi_{(\lambda,\infty)}
(|x|) \big)^{1/p}.$$
We refer the reader to \cite{FK,P2} for more information on noncommutative $L_p$-spaces.

\subsection{Noncommutative Hilbert valued $L_{p}$-spaces}
The noncommutative Hilbert valued $L_{p}$-spaces present a suitable framework for studying the square functions in the noncommutative setting.
Let $\mathcal{H}$ be a Hilbert space. Given $v \in \mathcal{H}$ with norm one, take $p_v=v\otimes \overline{v}$ the rank one projection onto span$\{v\}$. Then for $1\leq p\leq \infty$, we define
 $$L_p(\M; \mathcal{H}^{r})=(p_v\otimes1_{\M}) L_p(B(\mathcal{H})\overline\otimes\M)\;\textrm{ and }\;
 L_p(\M; \mathcal{H}^{c})= L_p(B(\mathcal{H})\overline\otimes\M)(p_v\otimes1_{\M}),$$
where $\1_{\mathcal{M}}$ stands for the unit
elements in $\mathcal{M}$ and
$B(\mathcal{H})$ is equipped with the usual trace. The definitions of these two spaces are essentially independent of the choice of $v$ (see \cite{JMX}). 
Therefore, we conclude that
$$\|u\|_{L_p(\M; \mathcal{H}^r)} = \big\| (uu^*)^{1/2}
\big\|_{L_p(\M)} \quad \mbox{and} \quad \|u\|_{L_p(\M;
\mathcal{H}^c)} = \big\| (u^*u)^{1/2} \big\|_{L_p(\M)}.$$

According to \cite[Chapter 2]{JMX}, we may apply these identities to regard $L_p(\M) \otimes \mathcal{H}$ as a dense subspace of $L_p(\M;\mathcal{H}^r)$ and $L_p(\M;\mathcal{H}^c)$. To be more specific, for $f = \sum_k u_k \otimes v_k \in
L_p(\M) \otimes \mathcal{H}$, we have
\begin{eqnarray*}
\|f\|_{L_p(\M;\mathcal{H}^r)} & = & \Big\| \Big( \sum_{i,j}
\langle v_i,v_j \rangle \, u_iu_j^* \Big)^{1/2}
\Big\|_{L_p(\M)}, \\ \|f\|_{L_p(\M;\mathcal{H}^c)} & = & \Big\|
\Big( \sum_{i,j} \langle v_i,v_j \rangle \, u_i^*u_j \Big)^{1/2}
\Big\|_{L_p(\M)}.
\end{eqnarray*}
This procedure can also be used to define
$$L_{1,\infty}(\M;\mathcal{H}^r)\ \ \mbox{and}\ \ L_{1,\infty}(\M;\mathcal{H}^c).$$

Finally, we define the mixture spaces $L_p(\mathcal{M};
\mathcal{H}^{rc})$ as follows:
$$L_p(\mathcal{M}; \mathcal{H}^{rc}) = \left\{ \begin{array}{ll}
L_p(\mathcal{M}; \mathcal{H}^{r}) + L_p(\mathcal{M}; \mathcal{H}^{c})
& 1 \le p \le 2, \\ L_p(\mathcal{M}; \mathcal{H}^{r}) \cap \hskip1pt
L_p(\mathcal{M}; \mathcal{H}^{c}) & 2 \le p \le \infty.
\end{array} \right.$$
It is obvious to see that
$L_2(\mathcal{M}; \mathcal{H}^{r}) = L_2(\mathcal{M};
\mathcal{H}^{c})=L_2(\mathcal{M};
\mathcal{H}^{rc})$.
The reader is referred to \cite{JMX} for a more general description of the Hilbert valued operator spaces.

\subsection{Noncommutative $\ell_{\infty}$-valued $L_{p}$-spaces and maximal inequalities}
Let $I$ be an index set.
Given $1\le p\leq\infty$, $L_p(\mathcal {M};\ell_\infty(I))$ consists of all families $(x_n)_{n\in I}$ in $L_p(\mathcal {M})$ which can be factorized as $x_n=ay_nb$ with $a,b\in L_{2p}( \mathcal {M})$  and $(y_n)_{n\in I}\subset L_\infty(\mathcal {M})$. The norm of $(x_n)_{n\in I}$ in $L_p(\mathcal {M};\ell_\infty(I))$ is defined as
$$\|(x_n)_{n\in I}\|_{L_p(\mathcal {M};\ell_\infty(I))}=\inf\left\{\big\|a\big\|_{{2p}}\sup_{n\in I}\big\|y_n\big\|_\infty\big\|b\big\|_{{2p}}\right\},$$
where the infimum is taken over all factorizations as above.
Usually, the norm of $(x_n)_{n\in I}$ in $L_p(\mathcal {M};\ell_\infty(I))$ is denoted by $\|{\sup_{n\in I}}^+x_n\|_p$, that is $\|{\sup_{n\in I}}^+x_n\|_p:=\|(x_n)_n\|_{L_p(\mathcal {M};\ell_\infty(I))}$.

The following property is stated in \cite[Remark 4.1]{CXY13}.
\begin{rk}\label{rk:MaxFunct}\rm
Let $(x_n)_{n\in I}$ be a sequence of selfadjoint operators in $L_p(\mathcal {M})$. Then $x=(x_n)_{n\in I}$ in $L_p(\mathcal {M};\ell_\infty(I))$ if and only if there is $a\in L_{p}(\M)_{+}$
such that $-a\leq x_n\leq a$ for all $n\in I$. Moreover, in this case,
$$\|x\|_{L_p(\mathcal {M}; \ell_\infty(I))}=\inf\Big\{\|a\|_p: a\in L_{p}(\M)_{+}\ \mbox{such that}\ -a\leq x_n\leq a
,\ \forall n\in I\Big\}.$$
\end{rk}

\noindent In the rest of this paper, we will omit the index set $I$ when it will not cause confusions.

It has already been shown in \cite{J1} that $L_{p}(\mathcal {M};\ell_\infty)$ is a dual space for every $p>1$ and its predual space is denoted by $L_{p'}(\mathcal {M};\ell_1)$ ($p'$ being the conjugate index of $p$). Let us briefly recall the latter space. Given $1\leq p\leq\infty$,
$L_{p}(\mathcal {M};\ell_1)$ is defined to be the space of all sequences $x=(x_n)_{n}$ in $L_{p}(\mathcal {M})$ which can be factorized as
$$x_n=\sum_{k}u_{kn}^{\ast}v_{kn}, \quad \forall n$$
for two families $(u_{kn})_{k,n}$ and $(v_{kn})_{k,n}$ in $L_{2p}(\mathcal {M})$ such that
$$\sum_{k,n}u^{\ast}_{kn}u_{kn}\in L_{p}(\mathcal {M})\;\; \text{and}\;\; \sum_{k,n}v^{\ast}_{kn}v_{kn}\in L_{p}(\mathcal {M}),$$
where all series are required to be convergent in $L_{p}(\mathcal {M})$ (relative to the w$^{*}$-topology for $p=\infty$).
The norm of $x$ in $L_{p}(\mathcal {M};\ell_1)$ is defined by
$$\|x\|_{L_{p}(\mathcal {M};\ell_1)}=\inf \|\sum_{k,n}u^{\ast}_{kn}u_{kn}\|_{p}^{1/2}\|\sum_{k,n}v^{\ast}_{kn}v_{kn}\|_{p}^{1/2},$$
where the infimum runs over all possible decompositions $x_n=\sum_{k}u_{kn}^{\ast}v_{kn}$. 
The duality between $L_{p}(\mathcal {M};\ell_{\infty})$ and $L_{p'}(\mathcal {M};\ell_1)$ is given by
$$\langle x,\;y\rangle=\sum_n\tau(x_ny_n),\quad x=(x_n)_{n}\in L_{p}(\mathcal {M};\ell_\infty),\;y=(y_n)_{n}\in L_{p'}(\mathcal {M};\ell_1).$$

The following properties  can be found in \cite[Proposition 2.1]{JX}.
\begin{prop}\label{JX2007}\rm
Let $1\leq p\leq\infty$.
\begin{itemize}
	\item [(i)] A sequence $x=(x_n)_{n\in I}$ in $L_{p}(\M)$ belongs to $L_p(\mathcal {M}; \ell_\infty(I))$ if and only if
	$$\sup\big\{\|{\sup_{n\in J}}^+x_n\|_p:\ J\subset I,\ J\ \mbox{is\ finite}\big\}<\infty.$$
	In this case, $\|{\sup_{n\in I}}^+x_n\|_p$ is equal to the above supremum. 

	\item [(ii)] Let $x=(x_n)_n$ be a positive sequence in  $L_p (\M; \ell_{\infty})$. Then
	$$ \big \|{\sup_n}^+ x_n \big \|_p = \sup \Big \{ \Big | \sum_n \tau(x_n y_n ) \ :\
	y_n \in L_{p'} (\M)_{+}\ \mbox{and}\  \Big \| \sum_ny_n \Big \|_{p'} \le 1 \Big \}.$$
\end{itemize}
\end{prop}

Based on these necessary notions, we can present the definition of the noncommutative maximal inequalities.
\begin{defi}\rm
Let $1\leq p\leq\infty$ and
let $S=(S_{n})_{n}$ be a family of maps from $L_{p}(\M)_{+}$ to $L_{0}(\M)_{+}$.
\begin{itemize}
\item [(i)] For $p<\infty$, we say that $S$ is of weak type $(p,p)$ with constant $C$ if there is a positive constant $C$ such that for any $x\in L_{p}(\M)_{+}$ and any $\lambda>0$, there is a projection $e\in\M$ satisfying
$$\forall n\ \ eS_{n}(x)e\leq\lambda\ \ \mbox{and}\ \ \tau(e^{\perp})\leq\frac{C^{p}\|x\|_{p}}{\lambda^{p}}.$$
\item [(ii)] For $1\leq p\leq\infty$, we say that $S$ is of strong type $(p,p)$ with constant $C$ if there is a positive constant $C$ such that for any $x\in L_{p}(\M)_{+}$ there exists $a\in L_{p}(\M)_{+}$ satisfying
$$\forall n\ \ S_{n}(x)\leq a\ \ \mbox{and}\ \ \|a\|_{p}\leq C\|x\|_{p}.$$
\end{itemize}
\end{defi}

We refer the reader to \cite{J1} and  \cite{JX} for more details.
\subsection{Almost uniform convergence}
In this subsection, we recall the noncommutative analogue of the usual almost everywhere convergence. The following definition is introduced by Lance \cite{Lan76}.
\begin{defi}\label{b.a.u.}\rm
Let $\mathcal {M}$ be a semifinite von Neumann algebra equipped with a semifinite normal faithful trace $\tau$. Let $x_n,x\in L_0(\M)$.
\begin{itemize}
\item[(i)]  $(x_n)$ is said to converge bilaterally almost uniformly (b.a.u. in short) to $x$ if for any $\varepsilon>0$, there is a
projection $e\in \M$ such that
$$\tau(e^\bot)<\varepsilon \ \ \  \mbox{and} \ \ \ \lim _{n\rightarrow\infty}\|e(x_n-x)e\|_\infty=0.$$

\item[(ii)]  $(x_n)$ is said to converge almost uniformly (a.u. in short) to $x$ if for any $\varepsilon>0$, there is a projection $e\in \M$
such that
$$\tau(e^\bot)<\varepsilon \ \ \  \mbox{and} \ \ \ \lim _{n\rightarrow\infty}\|(x_n-x)e\|_\infty=0.$$
\end{itemize}
\end{defi}
\subsection{Operator-valued Hardy spaces and $\mathrm{BMO}$ spaces}
In this subsection,
let us start by introducing Mei's notion \cite{M} of row and column Hardy spaces. According to our requirement, here we only concentrate on operator-valued $\mathrm{H}_1$ space; for a general description of operator-valued Hardy spaces we refer the reader to \cite{M}. Define
$$\mathrm{H}_1(\R^d; \M)  =  \mathrm{H}_1^r(\R^d; \M) + \mathrm{H}_1^c(\R^d;\M)$$ equipped with the following sum norm $$\|f\|_{\mathrm{H}_1(\R^d;\M)}  = \inf_{f = g+h} \|g\|_{\mathrm{H}_1^r(\R^d;\M)} + \|h\|_{\mathrm{H}_1^c(\R^d;\M)} < \infty,$$ where the row/column norms are given  by
\begin{eqnarray*}
\|g\|_{\mathrm{H}_1^r(\R^d;\M)}  = \Big\| \Big( \int_\Gamma \Big[ \frac{\partial P_g}{\partial t} \hskip0.5pt \frac{\partial P_g^*}{\partial t} \hskip0.5pt + \hskip0.5pt \sum_j \frac{\partial P_g}{\partial x_j} \hskip0.5pt \frac{\partial P_g^*}{\partial x_j} \Big] (x + \cdot,t) \, \frac{dx dt}{t^{n-1}} \Big)^{1/2} \Big\|_1, \\ [-1pt] \|h\|_{\mathrm{H}_1^c(\R^d;\M)}  =  \Big\| \Big( \int_\Gamma \Big[ \frac{\partial P_h^*}{\partial t} \frac{\partial P_h}{\partial t} + \sum_j \frac{\partial P_h^*}{\partial x_j} \frac{\partial P_h}{\partial x_j} \Big] (x + \cdot,t) \, \frac{dx dt}{t^{n-1}} \Big)^{1/2} \Big\|_1,
\end{eqnarray*}
with $\Gamma = \{ (x,t) \in \R^{d+1}_+ \, | \ |x| < y \}$ and $P_f(x,t) = P_t f(x)$ stands for the Poisson semigroup $(P_t)_{t \ge 0}$. 
We say that $a \in L_1(\M; L_2^c(\R^d))$ is a \emph{column atom} if there is a cube $Q$ such that
\begin{itemize}
\item [(i)] $\mathrm{supp}_{\R^d} \hskip1pt a = Q$;

\item [(ii)] $\displaystyle \int_Q a(y) \, dy = 0$;

\item [(iii)] $\|a\|_{L_1(\M;L_2^c(\R^d))} = \displaystyle \tau \Big[ \big( \int_Q |a(y)|^2 \, dy \big)^{\frac{1}{2}} \Big] \le |Q|^{-\frac{1}{2}}$, where $|Q|$ is the volume of $Q$.
\end{itemize}

It is known from \cite[Theorem 2.8]{M} that $$\|f\|_{\mathrm{H}_1^c(\R^d;\M)} \, \thickapprox \, \inf \Big\{ \sum_n |\lambda_n| \, \big| \ f = \sum_n \lambda_n a_n \ \mbox{with} \ a_n \ \mbox{column atoms} \Big\}.$$

The operator-valued $\mathrm{BMO}$ spaces are also studied in \cite{M}. Let $Q$ be a cube in $\R^d$ with sides parallel to the axes. 
For a function $f:\R^{d}\rightarrow\M$ is integrable on $Q$, $f_Q$ denotes its average over $Q$, that is
$$f_Q = \frac{1}{|Q|} \int_Q f(x)dx.$$
Recall that $\mathcal N=L_{\infty}(\R^{d})\overline{\otimes}\M$ given in the introduction. The $\mathrm{BMO}$ space $\mathrm{BMO}(\mathcal N)$ is defined as a subspace of $L_{\infty}(\M;L_2^{rc}(\R^{d};dx/(1+|x|)^{d+1}))$ with
$$\|f\|_{\mathrm{BMO}(\mathcal N)} \, = \, \max \Big\{
\|f\|_{\mathrm{BMO}^r(\mathcal N)}, \|f\|_{\mathrm{BMO}^c(\mathcal N)}
\Big\} \, < \, \infty,$$ where the row/column $\mathrm{BMO}$
norms are given by
\begin{eqnarray*}
	\|f\|_{\mathrm{BMO}^r(\mathcal N)} & = & \sup_{Q\subset \R^{d}} \Big\| \Big(
	\frac{1}{|Q|} \int_Q \big( f(x) - f_Q \big) \big( f(x) - f_Q
	\big)^{\ast} \, dx \Big)^{1/2} \Big\|_{\M}, \\
	\|f\|_{\mathrm{BMO}^c(\mathcal N)} & = & \sup_{Q\subset\R^{d}} \Big\| \Big(
	\frac{1}{|Q|} \int_Q \big(f(x) - f_Q \big)^{\ast} \big(f(x) - f_Q \big)
	\, dx \Big)^{1/2} \Big\|_{\M}.
\end{eqnarray*}

In the sequel, we will frequently employ the following Cauchy-Schwarz type inequality (see \cite[(1.13)]{M}). 
Let $(\Omega,\mu)$ be a measure space. Then
\begin{equation}\label{CS}
 \Big|\int_{\Omega}\phi fd\mu\Big|^2\le \int_{\Omega}|\phi|^2d\mu \int_{\Omega} |f|^2d\mu,
\end{equation}
where $\phi:\Omega\to\C$ and $f:\Omega\to L_1(\M)+L_\infty(\M)$ are operator-valued functions such that all members of the above inequality make sense.
\section{Bochner-Riesz means for Hermite operator}
In this section, we are concerned with convergence properties of operator-valued Bochner-Riesz means for Hermite operator $H$, which is defined by 
$$H=-\Delta+|x|^{2}=-\sum_{n=1}^{d}\frac{\partial^{2}}{\partial x_{n}^{2}}+|x|^{2},\ \ x=(x_{1},...,x_{d}).$$
For nonnegative integer $n$, the Hermite functions $H_{n}(t)$ on $\R$ are defined by  $H_{n}(t)=(-1)^{n}\exp(-t^{2})(\frac{d}{dt})^{n}\{\exp(-t^{2})\}$, and 
the normalised Hermite functions $\phi_{n}(t)$ are then given by
\begin{equation}\label{899}
\phi_{n}(t)=(2^{n}\sqrt{\pi}n!)^{-1/2}\exp(-t^{2}/2)H_{n}(t),\ n=0,1,2,...,
\end{equation}
which form a complete orthonormal  system in $L_{2}(\R)$. For every multi-index $\nu=(\nu_{1},\nu_{2},...,\nu_{d})$ and $x\in\R^{d}$, the $d$-dimensional Hermite functions $\Phi_{\nu}(x)$ are defined by taking the tensor product of one dimensional normalised Hermite function 
\begin{equation}\label{8919}
	\Phi_{\nu}(x)=\phi_{\nu_{1}}(x_{1})\phi_{\nu_{2}}(x_{2})...\phi_{\nu_{d}}(x_{d}),\ \ x=(x_{1},...,x_{d}).
\end{equation}
Then the functions $\{\Phi_{\nu}\}_{\nu}$ are eigenfunctions for the Hermite operator $H$ with eigenvalue $(2|\nu|+d)$, where $|\nu|=\nu_{1}+\nu_{2}+...+\nu_{d}$; and form a complete orthonormal  system in $L_{2}(\R^{d})$. Hence, for every $f\in L_{2}(\mathcal{N})$, we have the Hermite expansion
\begin{equation}\label{991}
	f(x)=\sum_{\nu}\widehat{f}(\nu)\Phi_{\nu}(x)= \sum_{n=0}^{\infty}P_{n}f(x),
\end{equation}
where $\widehat{f}(\nu)$ is  defined by
$\widehat{f}(\nu) = \int_{\R^d}f(x)\Phi_{\nu}(x) \, dx$
and $P_{n}$ denotes the Hermite projection given by
\begin{equation}\label{99232}
	P_{n}f(x)=\sum_{|\nu|=n}\widehat{f}(\nu)\Phi_{\nu}(x).
\end{equation}

For $R>0$, the Bochner-Riesz means for $H$ of order $\alpha\geq0$ is defined by
\begin{equation}\label{993}
	S_{R}^{\alpha}f(x)=\sum_{n=0}^{\infty}\Big(1-\frac{2n+d}{R}\Big)_{+}^{\alpha}P_{n}f(x),
\end{equation}
where $f$ is a $L_1(\M)\cap L_\infty(\M)$-valued compactly supported measurable function.
\subsection{Mean convergence} In this subsection, we  study the $L_p$ convergence of Bochner-Riesz means. First of all, we list some basic estimates of Bochner-Riesz kernel obtained by Thangavelu \cite{T1,T3}. Denote by $S_{R}^{\alpha}(x,y)$ the Bochner-Riesz kernel associated to the operator $S_{R}^{\alpha}$. Then it is not difficult to verify that
\begin{equation} \label{4}
	S_{R}^{\alpha}(x,y)=\sum_{n=0}^{\infty}\Big(1-\frac{2n+d}{R}\Big)_{+}^{\alpha}\Phi_{n}(x,y),
\end{equation}
where the function $\Phi_{n}(x,y)$ is defined as
$\Phi_{n}(x,y)=\sum_{|\nu|=n}\Phi_{\nu}(x)\Phi_{\nu}(y).$

\begin{lem}[\cite{T1}]\label{14}\rm
For $d=1$ and $\alpha>1/6$, the following estimate is valid:
$$
|S_{R}^{\alpha}(x,y)|\leq CR^{1/2}\big\{(1+R^{1/2}|x-y|)^{-\alpha-5/6}
+(1+R^{1/2}|x+y|)^{-\alpha-5/6}\big\},
$$
where the constant $C$ is independent of $x,y$ and $R$.
\end{lem}
\begin{lem}[\cite{T3}]\label{15}\rm
If $1\leq p\leq2$, $d\geq2$ and $\alpha>\frac{d-1}{2}$, then for any $r>0$, the following estimate is valid:
$$\Big(\int_{|x-y|\geq r}|S_{R}^{\alpha}(x,y)|^{p}dy\Big)^{1/p}\leq CR^{d/2q}(1+R^{1/2}r)^{-\alpha-1/2+d(1/p-1/2)}$$
where $1/p+1/q=1$ and the constant $C$ is independent of $R,r$ and $x$.
\end{lem}

Now we have all ingredients to study the mean convergence of Bochner-Riesz means. 
\begin{thm}\label{tc}\rm
Let $1\leq p<\infty$ and $f\in L_{p}(\mathcal{N})$. Then
 \begin{itemize}
 \item[(i)] for $d=1$ and $\alpha>1/6$,
\begin{equation*}
\|S_{R}^{\alpha}f-f\|_{p}\rightarrow0\  \mbox{as}\  R\rightarrow\infty;
\end{equation*}

\item [(ii)] for $d\geq2$ and $\alpha>\frac{d-1}{2}$,
\begin{equation*}
\|S_{R}^{\alpha}f-f\|_{p}\rightarrow0\  \mbox{as}\  R\rightarrow\infty.
\end{equation*}
\end{itemize}
\end{thm}
\begin{proof}
By decomposing $f = f_{1}-f_{2} +i(f_{3}-f_{4})$ with positive $f_{k}$ ($k=1,2,3,4$),
we can assume $f$  is positive.
By Lemma \ref{14}, 
\begin{equation}\label{5}
-C\{E_{R}(x-y)+E_{R}(x+y)\}\leq S_{R}^{\alpha}(x,y)\leq C\{E_{R}(x-y)+E_{R}(x+y)\}
\end{equation}
where $E_{R}(x)=R^{1/2}(1+R^{1/2}|x|)^{-\alpha-5/6}$. 
Set $\widetilde{f}(x)=f(-x)$ and $E_{R}f(x)=f\ast E_{R}(x)$. Then by (\ref{5}), we deduce that
\begin{equation} \label{6}
-C\{E_{R}f(x)+E_{R}\widetilde{f}(x)\}\leq S_{R}^{\alpha}f(x)\leq C\{E_{R}f(x)+E_{R}\widetilde{f}(x)\}.
\end{equation}

Note that $E_{R}(x)$ is an $L_{1}(\R)$ function when $\alpha>1/6$. Thus, by Young's inequality, it follows immediately that $\|E_{R}f\|_{p}\lesssim\|f\|_{p}$ and $\|E_{R}\widetilde{f}\|_{p}\lesssim\|f\|_{p}$. Therefore, we obtain the uniform boundedness of $S_{R}^{\alpha}$ thanks to (\ref{6}). On the other hand,  $S_{R}^{\alpha}f$ converges to $f$ in $L_{p}$-norm whenever $f\in C_{c}^{\infty}(\R)\otimes \mathcal{S}_{\M}$ (see e.g. \cite[Exercise 6.2.9]{Gra2008}). Consequently, a density argument implies $S_{R}^{\alpha}f$ converges to $f$ in $L_{p}$-norm. This proves (i).

We now turn to (ii). Letting $r\rightarrow0$ in Lemma \ref{15} with $p=1$, we see that $S_{R}^{\alpha}(x,y)$ are uniformly integrable for $\alpha>\frac{d-1}{2}$. This gives the uniform boundedness of $S_{R}^{\alpha}$. Together with the density argument, we complete the proof.
\end{proof}
\subsection{Pointwise convergence}
In this subsection, we study the pointwise convergence of the Bochner-Riesz means by showing the corresponding noncommutative maximal inequalities of the sequence $(S_{R}^{\alpha})_{R>0}$. 
Before it, we present some necessary lemmas. The first one is well-known to experts, its proof can be found in \cite[Theorem 4.3]{CXY13}.
\begin{lem}[\cite{CXY13}]\label{result of cxy}\rm
Let $\psi$ be an integrable function on $\R^d$ such that $|\psi|$ is radial and radially decreasing.
Let $\psi_t(x)=\frac1{t^d}\, \psi(\frac xt)$ for $x\in\R^d$ and $t>0$.
\begin{itemize}
 \item [(i)] Let $f \in L_{1}(\mathcal{N})$. Then for  any $\lambda>0$ there exists a projection $e\in \mathcal{N}$ such that
 $$\sup_{t>0}\big\|e(\psi_t\ast f)e\big\|_\infty\leq\lambda\quad\text{and}\quad \varphi (e^{\perp})\leq C_d\|\psi\|_1\frac{\|f\|_1}\lambda,$$
where $e^{\perp}=1_\n-e$.

\item [(ii)] Let $1<p\leq\infty$. Then
 $$\big\|{\sup_{t>0}}^+\psi_t*f\big\|_p\le C_d\|\psi\|_1\,\frac{p^2}{ (p-1)^2}\|f\|_p,\quad\forall\; f\in L_p(\mathcal{N}).$$
\end{itemize}
\end{lem}

The second lemma connecting Bochner-Riesz means of different order is useful.
\begin{lem}\label{thu}\rm
	Let $\beta,\delta$ be two complex numbers such that $\mathrm{Re}\beta>0$, $\mathrm{Re}\delta>-1$ and $\mathrm{Re}(\beta+\delta)>0$. Then
	\begin{equation*} 
	S_{R}^{\delta+\beta}=\frac{\Gamma(\delta+\beta+1)}{\Gamma(\delta+1)\Gamma(\beta)}
	\int_{0}^{1}(1-t)^{\beta-1}t^\delta S_{Rt}^{\delta}dt.
	\end{equation*}
\end{lem}
\begin{proof}
By definition, it is sufficient to verify
	\begin{equation*}
	\Big(1-\frac{N}{R}\Big)^{\delta+\beta}=\frac{\Gamma(\delta+\beta+1)}{\Gamma(\delta+1)\Gamma(\beta)}
	\int_{N/R}^{1}(1-t)^{\beta-1}t^\delta\Big(1-\frac{N}{Rt}\Big)^\delta dt
\end{equation*}
for $\mathrm{Re}\beta>0$, $\mathrm{Re}\delta>-1$ and $\mathrm{Re}(\beta+\delta)>0$, where $N=2n+d$. 
The remaining argument is quite similar as in
\cite[Lemma 4]{Stein1958}.
\end{proof}

\begin{thm}\label{t6}\rm
Let $S_{R}^{\alpha}$ be defined in (\ref{993}) with $d=1$ and $\alpha>1/6$. 
\begin{enumerate}
\item [(i)] Let $f \in L_{1}(\mathcal{N})$. Then $(S_{R}^{\alpha})_{R>0}$ is of weak type $(1,1)$, that is for any $\lambda>0$ there exists a projection $e\in \mathcal{N}$ satisfying
$$\| eS_{R}^{\alpha}fe\|_{\infty}\lesssim\ \lambda\ \mbox{for\ all}\ R>0\ \ \ \mbox{and} \ \ \ \varphi (e^{\perp})\lesssim \frac{\| f\|_{1}}{\lambda}.$$

\item [(ii)] Let $1<p\leq\infty$. Then
$$\|\sup_{R>0}\!^{+}S_{R}^{\alpha}f\|_{p}\lesssim\| f\|_{p},\ \forall f\in L_{p}(\mathcal{N}).$$

\item [(iii)] For any $f\in L_{p}(\mathcal{N})$ with $1\leq p<\infty$,
$$S_{R}^{\alpha}f\xrightarrow{\rm b.a.u}f\  \mbox{as}\  R\rightarrow\infty.$$
\end{enumerate}
\end{thm}
\begin{proof}
(i) 
Without loss of generality, we may assume $f$ is positive. Note that the Bochner-Riesz kernel $S_{R}^{\alpha}(x,y)$ is real-valued since $\alpha$ is positive; moreover, 
from (\ref{6}) in the proof of Theorem \ref{tc}, we know that
\begin{equation} \label{346}
-(E_{R}f+E_{R}\widetilde{f})\lesssim S_{R}^{\alpha}f\lesssim E_{R}f+E_{R}\widetilde{f}.
\end{equation}

If we set $E(x)=(1+|x|)^{-\alpha-5/6}$, then $E_{R}(x)=R^{1/2}E(R^{1/2}x)$.
It is clear that $E$ is an integrable function on $\R^{d}$. Moreover, $E$ is radial and radially decreasing. Therefore, by Lemma \ref{result of cxy} (i), there exists a projection $e_{1}\in\mathcal{N}$ such that for any $R>0$
$$\varphi(e_{1}^{\perp})\lesssim \frac{\|f\|_{1}}{\lambda}\ \ \mbox{and}\ \ -\frac{\lambda}{2}\leq e_{1} E_{R}fe_{1}\leq\frac{\lambda}{2}.$$

Similarly, we can find a projection $e_{2}\in\mathcal{N}$ such that for any $R>0$
$$\varphi(e_{2}^{\perp})\lesssim \frac{\|\widetilde{f}\|_{1}}{\lambda}=\frac{\|f\|_{1}}{\lambda}\ \ \mbox{and}\ \ -\frac{\lambda}{2}\leq e_{2} E_{R}\widetilde{f}e_{2}\leq\frac{\lambda}{2}.$$

Set $e=e_{1}\wedge e_{2}$. We then deduce that for any $R>0$,
$$-\lambda\leq -(e_{1}E_{R}fe_{1}+e_{2}E_{R}\widetilde{f}e_{2})\lesssim  eS_{R}^{\alpha}fe\lesssim e_{1}E_{R}fe_{1}+e_{2}E_{R}\widetilde{f}e_{2}\leq\lambda,$$
which implies 
$$\sup_{R>0}\|eS_{R}^{\alpha}fe\|_\infty\lesssim\lambda\ \ \mbox{and}\ \ \varphi(e^{\perp})\leq\varphi(e_{1}^{\perp})+\varphi(e_{2}^{\perp})\lesssim \frac{\|f\|_{1}}{\lambda}.$$
This establishes (i).

To prove part (ii), just by noting (\ref{346}), Remark \ref{rk:MaxFunct}, triangle's inequality in $L_p(\n;\ell_\infty)$ and Lemma \ref{result of cxy} (ii), we find
\begin{align*}
	\|{\sup_{R>0}}^{+}S_{R}^{\alpha}f\|_{p}
	\lesssim&\ \|{\sup_{R>0}}^{+}E_{R}f+E_{R}\widetilde{f}\|_{p}\\
	\leq&\ \|{\sup_{R>0}}^{+}E_{R}f\|_{p}+\|{\sup_{R>0}}^{+}E_{R}\widetilde{f}\|_{p}\lesssim\|f\|_{p}.
\end{align*}

(iii) The pointwise convergence results can be obtained as a byproduct of the previous maximal inequalities through a standard verification, see e.g. \cite[Section 7]{HLX}. So we omit the proof.
\end{proof}
Let us now consider a deeper insight into the one-dimensional case. If $\alpha$ is smaller than the critical index $1/6$, Theorem \ref{t6} (ii) usually fail even in the scalar case (see \cite{T1}). However, we can improve this result by proving Theorem \ref{local} below, which is the noncommutative analogue of Stein's theorem \cite{Stein1958} for Bochner-Riesz means associated to Hermite operator. 
\begin{thm}\label{local}\rm
Assume that $d=1$ and $\alpha>1/3\big|1/2-1/p\big|$. Then for any $f\in L_{p}(\mathcal{N})$ with $1< p<\infty$, we have

\begin{enumerate}

\item [(i)] $\| {\sup\limits_{R>0}}^{+}S_{R}^{\alpha}f\|_{p}\lesssim \| f\|_{p}$, where the implicit constant depends on $p$, $d$ and $\alpha$.

\item [(ii)] $\|S_{R}^{\alpha}f-f\|_{p}\rightarrow0\  as\ R\rightarrow\infty.$

\item [(iii)] $S_{R}^{\alpha}f\xrightarrow{\rm b.a.u}f\  \mbox{as}\  R\rightarrow\infty$.
\end{enumerate}
\end{thm}
\begin{proof}
It suffices to prove the maximal inequality (i). Indeed, by Remark \ref{rk:MaxFunct}, part (i) implies that 
$$\sup_{R>0}\|S_{R}^{\alpha}f\|_{p}\leq\| \sup_{R>0}\!^{+}S_{R}^{\alpha}f\|_{p}\lesssim\| f\|_{p}.$$
Together with the density of $C_{c}^{\infty}(\R^{d})\otimes \mathcal{S}_{\M}$ in $L_{p}(\mathcal{N})$, we obtain (ii). On the other hand, conclusion (iii) can be proved by (i) via a standard verification. Thus the remainder  is devoted to the proof of part (i) and the idea is inspired by  Stein in classical setting \cite{Stein1958} and Chen et al on quantum tori \cite{CXY13}. 

By Proposition \ref{JX2007} (i), it suffices to show for any finite subset $J\subset(0,\infty)$,
\begin{equation} \label{202}
	\| \sup_{R\in J}\!^{+}S_{R}^{\alpha}f\|_{p}\lesssim\| f\|_{p},\ \ \forall f\in L_{p}(\mathcal{N}).
\end{equation}

In the following, we fix $J$ and $f\in L_{p}(\mathcal{N})$.
For clarity we divide the proof of (\ref{202}) into three steps.

\medskip\noindent {\it Step 1}. For $\alpha\in\C$, $\mathrm{Re}\alpha>1/6$ and $1<p\leq\infty$, we prove that
\begin{equation} \label{2012}
\| \sup_{R\in J}\!^{+}S_{R}^{\alpha}f\|_{p}\lesssim\| f\|_{p}.
\end{equation}
To see this, choose $\delta>0$ and $\beta\in\C$ such that $\mathrm{Re}\alpha>\delta>1/6$ and $\alpha=\delta+\beta$. By Lemma \ref{thu}, we have the following equality:
\begin{equation} \label{20}
S_{R}^{\alpha}=C_{\beta,\delta}\int_{0}^{1}(1-t)^{\beta-1}t^\delta S_{Rt}^{\delta}dt,
\end{equation}
where $C_{\beta,\delta}=\Gamma(\beta+\delta+1)/\Gamma(\delta+1)\Gamma(\beta).$
Since $\mathrm{Re}\beta=\mathrm{Re}\alpha-\delta>0$ and $\delta>0$, we have
\begin{equation}\label{2001}
	\int^1_0 |(1 - t)^{\beta -1}t^\delta|dt=\int^1_0(1 - t)^{Re\beta -1}t^\delta dt<\infty.
\end{equation}

Therefore, combining (\ref{2001}) with 
Theorem \ref{t6} (ii) and triangle's inequality in $L_p(\n;\ell_\infty)$, 
we conclude that
$$
\| \sup_{R\in J}\!^{+}S_{R}^{\alpha}f\|_{p}
\leq |C_{\beta, \delta}|\int^1_0 (1 - t)^{\mathrm{Re}\beta -1}t^\delta dt \big \| {\sup\limits_{R\in J}}^+ S_R^{\delta}(f) \big \|_p\\
\lesssim  \| f \|_p.
$$

\medskip

\noindent{\it Step 2}.For $\alpha\in\C$ with $\mathrm{Re}\alpha>0$, we show that
\begin{equation}\label{2013}
	\big \| {\sup_{R\in J}}^{+}S_R^{\alpha}(f)\big\|_2 \lesssim \|f\|_2.
\end{equation}
Indeed, by choosing $\delta>0$, $\beta\in\C$ such that $\mathrm{Re}\alpha>\delta>0$ with $\alpha=\delta+\beta$ and using the same argument in {\it Step 1}, we can reduce to proving (\ref{2013}) when $\alpha>0$.

We first consider the case $ \alpha>1/2$. Choose
$\beta > 1$ and $\delta > -1/2$ such that $\alpha = \beta + \delta.$ By Lemma \ref{thu} and integration by parts,
$$S_{R}^{\beta+\delta}=C_{\beta,\delta}\int_{0}^{1}\psi(t)M^\delta_{Rt}dt,$$
where $M^\delta_{t}=\frac{1}{t}\int_{0}^{t}S_{r}^{\delta}dr$
and $\psi(t)=(\beta-1)(1-t)^{\beta-2}t^{\delta+1}-\delta(1-t)^{\beta-1}t^\delta.$
A simple calculation shows that $\int_{0}^{1}|\psi(t)|dt<\infty$. Hence, we  use again triangle's inequality in $L_p(\n;\ell_\infty)$, 
$$\big \| {\sup\limits_{R\in J}}^+ S_R^{\delta}(f) \big \|_2\leq|C_{\beta, \delta}|\int^1_0 |\psi(t)|dt\big \| {\sup\limits_{R\in J}}^+ M_R^{\delta}(f) \big \|_2 \lesssim\big \| {\sup\limits_{R\in J}}^+ M_R^{\delta}(f) \big \|_2.$$

Therefore, it suffices to show if $\delta > -1/2$
\begin{equation}\label{20153}
	\big \| {\sup\limits_{R\in J}}^+ M_R^{\delta}(f) \big \|_2\lesssim \|f\|_2.
\end{equation}
Let $(R_n)_{n}$ be any fixed sequence in $J$ and $(g_n)$ be a sequence of positive elements in $L_{2}(\mathcal{N})$ such that $\|\sum_ng_n\|_2\leq1$. Then by Proposition \ref{JX2007} (ii),
\begin{align*}
	\Big | \varphi \Big (  \sum_n M_{R_n}^{\delta}(f) g_n \Big ) \Big |
	\leq&\ \Big | \varphi \Big (  \sum_n M_{R_n}^{\delta+1}(f) g_n \Big ) \Big |+\Big | \varphi \Big (  \sum_n [M_{R_n}^{\delta+1}(f)-M_{R_n}^{\delta}(f)] g_n \Big ) \Big |\\
	\leq&\ \big \| {\sup\limits_{R\in J}}^+ M_R^{\delta+1}(f) \big \|_2+\Big | \varphi \Big (  \sum_n G_{R_n}^{\delta}(f) g_n \Big ) \Big |,
\end{align*}
where $G_{R_n}^{\delta}=M_{R_n}^{\delta+1}-M_{R_n}^{\delta}$.
In the following, we will need a fundamental inequality
(see e.g. \cite{CXY13})
\begin{equation}\label{2017}
|\varphi (ab)|^2 \le \varphi(|a| b ) \varphi (| a^*| b),\ \forall\ a,b\in\mathcal N\ \mbox{with}\ b\geq0.
\end{equation}
Then by above inequality (\ref{2017}) and Cauchy-Schwartz inequality, we see that
\begin{align*}
	\Big | \varphi \Big ( \sum_n G^{\delta}_{R_n} (f) g_n \Big )\Big |^2
	&\leq\Big (  \sum_n\big | \varphi\big( G^{\delta}_{R_n} (f) g_n \big )\big |\Big )^2\\
	&\leq \Big (  \sum_n \varphi\big( |G^{\delta}_{R_n} (f)| g_n \big )^{1/2}\varphi\big( |G^{\delta}_{R_n} (f)^*| g_n \big )^{1/2}\Big )^2\\
	&\leq \varphi\Big ( \sum_n |G^{\delta}_{R_n} (f)| g_n \Big ) \varphi \Big (\sum_n |G^{\delta}_{R_n} (f)^*| g_n \Big).
\end{align*}
Noticing that by the definition of $M_{R_n}^{\delta}$, (\ref{CS}) and the operator monotonicity of $0\leq x\mapsto x^t$ for $0<t<1$, we deduce that
\begin{align*}
	|G^{\delta}_{R_n} (f) |
	&=\Big|\frac{1}{R_n}\int^{R_n}_0[S_{r}^{\delta+1}(f)-S_{r}^{\delta}(f)]dr\Big|\le \Big ( \int^{R_n}_0 \big| S_r^{\delta+1}(f)-S_r^{\delta}(f)  \big|^2 \frac{d r}{R_n}\Big )^{1/2}\\
	&\leq \Big ( \int^{\infty}_0 \big | S_r^{\delta+1} (f) - S_r^{\delta}(f) \big |^2 \frac{ d r}{r}\Big )^{1/2}=:G^{\delta} (f).
\end{align*}
It then follows from Cauchy-Schwarz's inequality that
$$\varphi\Big ( \sum_n |G^{\delta}_{R_n} (f)| g_n \Big )\leq \varphi\Big (G^{\delta} (f)\sum_ng_n\Big)\leq\|G^{\delta} (f)\|_2.$$

On the other hand, the same argument gives
$$\varphi \Big (\sum_n |G^{\delta}_{R_n} (f)^*| g_n \Big)\leq\|G^{\delta}_* (f)\|_2, $$
where $G^{\delta}_*$ is defined by
$$G^{\delta}_* (f)=\Big( \int^{\infty}_0 \big |\big(S_r^{\delta+1}(f) - S_r ^{\delta}(f)\big)^*\big |^2\frac{ d r}{r}\Big )^{1/2}.$$

Combining above observations and duality, we get
$$\big \| {\sup\limits_{R\in J}}^+ M_R^{\delta}(f) \big \|_2\leq\big \| {\sup\limits_{R\in J}}^+ M_R^{\delta+1}(f) \big \|_2+\| G^{\delta} (f)\|^{1/2}_2 \|G^{\delta}_\ast (f)\|^{1/2}_2. $$

In the following, we claim that
$$\max\big\{\| G^{\delta} (f) \|_2,\; \| G^{\delta}_*(f) \|_2 \big\}\lesssim \| f \|_2.$$
Consider the term $\| G^{\delta} (f) \|_2$ firstly. Observe that Parseval's identity and Fubini's theorem imply
\begin{align*}
	\| G^{\delta} (f) \|_2^2
	& =\int^{\infty}_0  \varphi \big(\big|S_r^{\delta+1}(f)  - S_r^{\delta}(f) \big |^2\big) \frac{ d r}{r}\\
	& = \int^{\infty}_0 \sum_{N \le r} \Big | \Big ( 1 - \frac{N}{r} \Big )^{\delta +1} -
	\Big ( 1 - \frac{N}{r} \Big )^{\delta} \Big |^2\|\widehat{f}(n)\|_2^2\frac{d r}{r}\\
	& = \sum_{n \not=0}\|\widehat{f}(n)\|_2^2 \int^{\infty}_{N} \frac{N^{2}}{r^{2}}\Big ( 1 - \frac{N}{r} \Big )^{2\delta} \frac{d r}{ r},
\end{align*}
where $N=2n+1$. Note that
$$\int^{\infty}_{N} \frac{N^{2}}{r^{2}}\Big ( 1 - \frac{N}{r} \Big )^{2\delta} \frac{d r}{ r}=\int_{1}^{\infty}r^{-3}(1-r^{-1})^{2\delta}dr<\infty,$$
since $\delta>-1/2$. Hence,
$\| G^{\delta} f \|_2\lesssim \|f\|_2.$
The same argument works for the other term $\|G^{\delta}_{\ast}f\|_{2}$ by noting  $\|G^{\delta}_{\ast}f\|_{2}=\|G^{\delta}f^*\|_{2}$.
This is precisely the claim. 
Consequently, 
$$\big \| {\sup\limits_{R\in J}}^+ M_R^{\delta}(f) \big \|_2\leq\big \| {\sup\limits_{R\in J}}^+ M_R^{\delta+1}(f) \big \|_2+\|f\|_2\lesssim\|f\|_2,$$
where in the last inequality we used (\ref{2012}).

Now we deal with the general case $\alpha>0$. In this case, choose
$\beta > 1/2$ and $\delta > -1/2$ such that $\alpha = \beta + \delta.$ Note that Lemma \ref{thu} and a change  of variable give
$$S_{R}^{\beta+\delta}=C_{\beta,\delta}R^{-(\beta+\delta)}\int^{R}_0 (R - t)^{\beta -1}t^\delta S_t^{\delta}dt.$$
Therefore, we obtain
\begin{align*}
S_{R_n}^{\beta+\delta}  - \frac{C_{\beta,\delta}}{C_{\beta, \delta +1}} S_{R_n}^{\beta+\delta+1}
	&=C_{\beta,\delta}R_n^{-(\beta+\delta)}\int^{R_n}_0 (R_n - t)^{\beta -1}t^\delta
	\big(S_t^{\delta}-S_t^{\delta+1} \big)d t\\
	&\ \ \ + C_{\beta,\delta}R_n^{-(\beta+\delta)}\int^{R_n}_0 (R_n - t)^{\beta -1}t^\delta
	(1-R_n^{-1}t)S_t^{\delta+1}d t\\
	&=:I_{R_n}+II_{R_n}.
\end{align*}

By (\ref{CS}) and the operator monotonicity of $0\leq x\mapsto x^t$ for $0<t<1$, we have
\begin{align*}
	|I_{R_n} (f) |
	&\leq|C_{\beta,\delta}| R_n^{1/2}R_n^{-(\beta+\delta)}\Big(\int^{R_n}_0 |(R_n - t)^{\beta -1}t^\delta|^2
	 d t \Big )^{1/2}\\
	&\ \ \ \times  \Big ( \int^{R_n}_0 \big| S_t^{\delta}(f)-S_t^{\delta+1}(f)  \big|^2 \frac{d t}{R_n}\Big )^{1/2}\lesssim G^{\delta} (f),
\end{align*}
since  $\beta > 1/2$ and $\delta > -1/2$, the integral
$$R_n^{1-2(\beta+\delta)}\int^{R_n}_0 |(R_n - t)^{\beta -1}t^\delta|^2dt=\int_0^1(1-t)^{2\beta-2}t^{2\delta}dt<\infty.$$
Similarly, $|I_{R_n} (f)^*|\lesssim G^{\delta}_* (f)$. Hence, by (\ref{2017}) and duality, we have 
$$\big\| {\sup_{R\in J}}^+ I_R (f)\big\|_2\lesssim \| G^{\delta} (f)\|^{1/2}_2 \|G^{\delta}_\ast (f)\|^{1/2}_2\lesssim\|f\|_2. $$

To estimate $II_{R_n}$, using the same argument as in the case $\alpha>1/2$, we obtain
$$II_{R_n}=C_{\beta,\delta}\int_{0}^{1}\rho(t)M^\delta_{R_nt}dt,$$
where $\rho(t)=\beta(1-t)t^{\delta+1}-\delta(1-t)^\beta t^\delta.$ 
Note that $\beta > 1/2$ and $\delta > -1/2$, $\int_{0}^{1}|\rho(t)|dt<\infty$. Then $II_{R_n}$ can be dealt with as in the case $\alpha>1/2$. Therefore, we  conclude that
$\big\| {\sup_{R\in J}}^+ I_R (f)\big\|_2\lesssim\|f\|_2. $

Finally, combing all above observations, we get
$$
\big \| {\sup_{R\in J}}^{+}S_R^{\alpha}(f)\big\|_{2}
\leq \frac{|C_{\beta,\delta}|}{|C_{\beta, \delta +1}|} \big\|{\sup_{R\in J}}^{+}S_R^{\alpha+1}(f)\Big\|_{2} + \big\|{\sup_{R\in J}}^{+}I_R(f)\Big\|_{2}+\big\|{\sup_{R\in J}}^{+}II_R(f)\Big\|_{2} \lesssim \|f\|_{2},
$$
where in the last inequality we used Theorem \ref{t6} since $\alpha+1>1$. This completes the argument of {\it Step 2}.

\medskip\noindent{\it Step 3}. 
The
general case can be obtained by
Stein's complex interpolation. To see this, assume that $\| f
\|_{p} \leq1$ and let $g =
(g_n)$ be a finite sequence in $L_{p'} (\mathcal N)$  with $\| g \|_{L_{p'} (\mathcal N;\ell_1)} \leq1.$ We first consider the case $1<p<2$. For any fixed $\alpha >
1/3(1/p - 1/2)$, we can find $p_1>1$, $\alpha_0>0$ and $\alpha_1 >1/6$ such that
$$\alpha = (1-t) \alpha_0 + t\alpha_1\quad \text{and}\quad
\frac{1}{p} = \frac{1-t}{2} +\frac{t}{p_1}$$
for some $0<t<1.$ Define
$$ h(z) = u|f|^{\frac{p(1-z)}{2}+ \frac{p z}{p_1}}\quad z \in \mathbb{C},$$
where $f= u |f|$ is the polar decomposition of $f$.  On the other hand, by \cite[Proposition 2.5]{JX}, there exists a function $m =(m_n)_n$ continuous on the
strip $\{ z \in \mathbb{C}:\; 0 \le \mathrm{Re} (z) \le 1 \}$ and
analytic in the interior so that $m(t) = g$ and
 \begin{equation}\label{2023}
  \sup_{s \in \mathbb{R}} \max \big \{ \big \| m( \mathrm{i} s ) \big \|_{L_2(\mathcal N;\ell_1)},
 \big \| m( 1 + \mathrm{i} s ) \big\|_{L_{p_1'} (\mathcal N; \ell_1)} \big \} \leq1.
 \end{equation}

Fix a sequence $(R_n)\subset J$ and $\delta>0$. We define
 $$F(z)= \exp\big(\delta(z^2 - t^2)\big )\sum_n \varphi \big (S_{R_n}^{(1-z) \alpha_0 + z \alpha_1} [h(z)]m_n (z) \big ). $$
Then $F$ is a function analytic in the open
strip $\{z \in \mathbb{C}\;:\; 0 < \mathrm{Re} (z) < 1 \}.$ By
(\ref{2013}), for any $s\in\R$, we deduce that
\begin{align*}
 |F(\mathrm{i} s ) |
 &\le\exp\big ( -\delta( s^2 +t^2)\big) \big \| \big ( S_{R_n}^{\gamma_{1}}
 (h(\mathrm{i} s)) \big )_n \big \|_{L_2(\mathcal N; \ell_{\infty})}\big \| m (\mathrm{i} s) \big\|_{L_2 (\mathcal N;\ell_1)}\\
 &\lesssim \| h(\mathrm{i} s) \|_2\lesssim1,
\end{align*}
where $\gamma_{1}=\alpha_0 + \mathrm{i} s (\alpha_1-\alpha_0)$.
Similarly, by (\ref{2012}), we obtain
\begin{align*}
 |F(1+\mathrm{i} s ) |
 &\le\exp\big ( -\delta( s^2 +t^2-1)\big) \big \| \big ( S_{R_n}^{\gamma_{2}}
 (h(1+\mathrm{i} s)) \big )_n \big \|_{L_{p_{1}}(\mathcal N; \ell_{\infty})}\big \| m (\mathrm{i} s) \big\|_{L_{p^{\prime}_{1}} (\mathcal N;\ell_1)}\\
 &\lesssim \| h(1+\mathrm{i} s) \|_2\lesssim1,
\end{align*}
where $\gamma_{2}=\alpha_1 + \mathrm{i} s (\alpha_1-\alpha_0)$. Therefore, the maximum principle implies $| F(t)|  \lesssim 1$, that is for $f$ satisfying $\|f \|_{L_p(\mathcal N)} \leq1$
$$ \big|\varphi \big ( \sum_n S_{R_n}^{\alpha} (f) m_n \big ) \big | \lesssim 1.$$
By duality  and homogeneity, we then get
$$ \big\|{\sup_{R\in J}}^+ S_R ^{\alpha}(f) \big\|_p\lesssim \|f\|_p, \quad \forall\; f \in  L_p (\mathcal N).$$

The argument for the case $p>2$ is similar once we start by setting $p_1 =\infty.$ Thus we finish the proof.
\end{proof}
\begin{rk}\rm
We have given a slightly more general result by allowing $\alpha$ to be complex. In other words, Theorem \ref{local} remains true under the condition that $\mathrm{Re}(\alpha)>1/3\big|1/2-1/p\big|$ with $\alpha\in\C$ and $1< p<\infty$.
\end{rk}
In the following, we study the maximal inequalities for Bochner-Riesz means in higher dimension ($d\geq2$).
\begin{thm}\label{t111}\rm
Let $d\geq2$ and $\alpha>\frac{d-1}{2}$. Then for any $2<p<\infty$ and $f\in L_{p}(\mathcal{N})$, we have
\begin{enumerate}

\item [(i)] $\|{\sup\limits_{R>0}}^+S_{R}^{\alpha}f\|_{p}\leq C_{p}\| f\|_{p}$.

\item [(ii)] $S_{R}^{\alpha}f\xrightarrow{\rm b.a.u}f\  \mbox{as}\  R\rightarrow\infty$.
\end{enumerate}
\end{thm}
\begin{proof}
As in proving Theorem \ref{local}, it is enough to show conclusion (i). The idea comes from \cite{T3}. 
Let $f\in L_{p}(\mathcal{N})$. 
We may assume $f$ is positive. Fix $x\in\R^d$, set $f_{k}(y)=f(y)$ whenever $2^{k}\leq |x-y|\leq 2^{k+1}$ and $f_{k}(y)=0$ otherwise. Then
\begin{equation}\label{6003}
S_{R}^{\alpha}f(x)= \sum_{k=-\infty}^{\infty}S_{R}^{\alpha}f_{k}(x).
\end{equation}

By (\ref{CS}), we get
$$
S_{R}^{\alpha}f_{k}(x)
\leq\Big(\int_{|x-y|\geq 2^{k}}|S_{R}^{\alpha}(x,y)|^{2}dy\Big)^{1/2}\Big(\int_{2^{k}\leq|x-y|\leq2^{k+1} }f(y)^{2}dy\Big)^{1/2}.
 $$
Exploiting Lemma \ref{15} into the first term above and (\ref{6003}), we arrive at
\begin{equation*}
S_{R}^{\alpha}f(x)\lesssim \sum_{k=-\infty}^{\infty}R^{d/4}(1+R^{1/2}2^{k})^{-\alpha-1/2}2^{\frac{kd}{2}}
\Big(2^{-(k+1)d}\int_{|x-y|\leq 2^{k+1}}f(y)^{2}dy\Big)^{1/2}.
\end{equation*}
From \cite[Theorem 4.2]{T3}, we have the following estimate:
\begin{equation}\label{8765}
	G(R):=\sum_{k=-\infty}^{\infty}R^{d/4}2^{kd/2}
	(1+R^{1/2}2^{k})^{-\alpha-1/2}\lesssim1.
\end{equation}

On the other hand, since $p>2$, by applying Mei's noncommutative Hardy-Littlewood maximal type $(p,p)$ inequality \cite{M} to $f^{2}\in L_{p/2}(\mathcal{N})$, there exists a positive operator $F\in L_{p/2}(\mathcal{N})$ such that
$$\|F\|_{p/2}\le \|f^{2}\|_{p/2}\quad\mbox{and}\quad
M_{B}(f^{2})\le F,\; \forall\; \textrm{ball } B \textrm{ centered at } x,$$
where $M_{B} (f)(x) = \frac{1}{|B|} \int_{B} f(y) dy$.
As a consequence, by the monotone increasing property of $0\leq x\rightarrow x^{t}$ with $0<t<1$, we infer that for any $k\in\Z$,
$$\Big(2^{-(k+1)d}\int_{|x-y|\leq 2^{k+1}}f(y)^{2}dy\Big)^{1/2}\leq F(x)^{1/2}\ \mbox{and}\ \|F^{1/2}\|_{p}=\|F\|^{1/2}_{p/2}\lesssim\|f\|_{p}.$$

Finally, combining with  (\ref{8765}),  we find
\begin{equation*}
	-F(x)^{1/2}\lesssim-G(R)F(x)^{1/2}\lesssim S_{R}^{\alpha}f(x)\lesssim G(R)F(x)^{1/2}\lesssim F(x)^{1/2}.
\end{equation*}
which implies, by Remark \ref{rk:MaxFunct}, that
$$\| \sup_{R>0}\!^{+}S_{R}^{\alpha}f\|_{p}\lesssim\|F^{1/2}\|_{p}\lesssim\|f\|_{p}.$$
This completes the proof.
\end{proof}

As the same argument in the case $d=1$, we give a slightly more general result by allowing $\alpha$ to be complex for $d\geq2$.
We omit the proof here.
\begin{thm}\label{local1}\rm
Let $2< p<\infty$ and  $f\in L_{p}(\mathcal{N})$. . Then for any complex number $\alpha$ with $\mathrm{Re}\alpha>\frac{d-1}{2}$, we have
$$\mbox{(i)}\ \| {\sup\limits_{R>0}}^{+}S_{R}^{\alpha}f\|_{p}\lesssim \| f\|_{p};\ \mbox{(ii)}\ \|S_{R}^{\alpha}f-f\|_{p}\rightarrow0\  \mbox{as}\ R\rightarrow\infty;\ \mbox{(iii)}\ S_{R}^{\alpha}f\xrightarrow{\rm b.a.u}f\  \mbox{as}\  R\rightarrow\infty.$$
\end{thm}

\section{A Marcinkiewicz type multiplier theorem}
In this section, we study a Marcinkiewicz type multiplier theorem for Hermite expansion in noncommutative setting. 
For $f\in C_{c}^{\infty}(\R^d)\otimes \mathcal{S}_{\M}$, we define
\begin{equation}\label{1007}
 T_{\mu}f(x)=\sum_{n=0}^{\infty}\mu(2n+d)P_{n}f(x),
\end{equation}
where $\mu$ is a bounded function defined on the set of nonnegative integers.
Recall that the finite difference operators are defined inductively by
$$\delta\mu(N)=\mu(N+1)-\mu(N)$$
and for $n\geq1$, 
$$\delta^{n+1}\mu(N)=\delta^{n}\mu(N+1)-\delta^{n}\mu(N).$$

The following is our main result in this section.
\begin{thm}\label{555}\rm
Assume that the function $\mu$ satisfies the conditions
$$|\delta^{r}\mu(N)|\leq CN^{-r}$$
for any $r<d/2+1$. Let $T_{\mu}$ be defined as (\ref{1007}). Then for $f\in L_{p}(\mathcal N)$ with $1<p<\infty$,
$$\|T_{\mu}f\|_{p}\leq C_{p,\mu}\|f\|_{p}.$$
\end{thm}

The proof of Theorem \ref{555} depends on the mapping properties of $g$ and $g^{*}$ functions in the following subsection.
\subsection{Littlewood-Paley $g$ function}
In this subsection, we develop a Littlewood-Paley-Stein theory for $g$ function defined by Hermite semigroup $H^{t}=e^{-tH}$ for $t>0$. 
These operators are defined by
\begin{equation}\label{109}
H^{t}f(x)=\sum_{n=0}^{\infty}e^{-Nt}P_{n}f(x)
\end{equation}
where $N=2n+d$ and they have the kernel
$$k_{t}(x,y)=\sum_{\nu} e^{-Nt}\Phi_{\nu}(x)\Phi_{\nu}(y).$$
In view of the Mehler's formula, the kernel $k_{t}$ is explicitly given by (see \cite{T4})
\begin{equation}\label{1024}
k_{t}(x,y)=(\sinh 2t)^{-d/2}e^{\Psi_t(x,y)},
\end{equation}
where
$$\Psi_t(x,y)=-\frac{1}{2}(|x|^{2}+|y|^{2})\coth2t+x\cdot y\frac{1}{\sinh2t}.$$

The Littlewood-Paley $g$ function is defined by
$$g^{c}(f)(x)=\Big(\int_{0}^{\infty}|\partial_{t}H^{t}f(x)|^{2}tdt\Big)^{1/2}.$$

To establish the Littlewood-Paley-Stein theory for Hermite semigroup, there are many difficulties to adapt the argument in \cite{S2,XXX16} to the present case, since the Hermite semigroup fails to satisfy the nice condition $H^{t}1=1$ for any $t>0$. Fortunately, with the help of the explicit form of the kernel $k_{t}(x,y)$, we can use the noncommutative Hilbert-valued Calder\'{o}n-Zygmund theory. To be more precise, for $t>0$, consider $\partial_{t}H^{t}$ as a singular integral operator where its associated kernel, denoted by $\partial_{t}k_{t}(x,y)$, takes values in Hilbert space $L_{2}(\R_{+};tdt)$. As this Hilbert space will appear frequently later, to simply notation, we denote $L_{2}(\R_{+};tdt)$ by $R_{d}$. 
\begin{thm}\label{main}\rm
	Let $1\leq p\leq \infty$. The following assertions hold:
	\begin{enumerate}[\noindent]
		\item\emph{(i)}~for $p=1$ and $f\in L_{1}(\mathcal N)$,
		$$\|\partial_{(\cdot)}H^{(\cdot)}f\|_{L_{1,\infty}(\mathcal N;R_{d}^{rc})}\lesssim\|f\|_{1};$$
		
		\item\emph{(ii)}~for $p=\infty$ and $f\in L_{\infty}(\mathcal N)$,
		$$\|
		\partial_{(\cdot)}H^{(\cdot)}f
		\|_{\mathrm{BMO}_{d}(\mathcal{N};R_{d}^{r})} +
		\|
		\partial_{(\cdot)}H^{(\cdot)}f
		\|_{\mathrm{BMO}_{d}(\mathcal{N};R_{d}^{c})} \lesssim
		\|f\|_\infty;$$
		\item\emph{(iii)}~for $1<p<\infty$ and $f\in L_{p}(\mathcal N)$,
		$$\|\partial_{(\cdot)}H^{(\cdot)}f\|_{L_{p}(\mathcal N;R_{d}^{rc})}\approx\|f\|_{p}.$$
	\end{enumerate}
\end{thm}

In order to use the noncommutative singular integral theory (see e.g. \cite{MP,JP1,HLX}), we need to verify the following two lemmas.
\begin{lem}\label{1052}\rm
	For $f\in L_{2}(\mathcal N)$, we have
	$$\|\partial_{(\cdot)}H^{(\cdot)}f\|_{L_{2}(\mathcal N;R_{d}^{rc})}\approx\|f\|_{2}.$$
\end{lem}
\begin{proof}
The examination of the $L_{2}$ boundedness is easy. Indeed, by the definition of Hermite semigroup, it follows that
$$\|\partial_{(\cdot)}H^{(\cdot)}f\|^{2}_{L_{2}(\mathcal N;R_{d}^{rc})}
=\tau\int_{\R^{d}}\int_{0}^{\infty}|\partial_{t}H^{t}f(x)|^{2}tdtdx.$$

Observe that for $t>0$,
$$\partial_{t}H^{t}f(x)=-\sum_{n=0}^{\infty} e^{-Nt}NP_{n}f(x),$$
where $N=2n+d$. Thus Fubini's theorem and Parseval's identity imply
$$\|\partial_{(\cdot)}H^{(\cdot)}f\|^{2}_{L_{2}(\mathcal N;R_{d}^{rc})}
=\sum_{n=0}^{\infty}\int_{0}^{\infty}e^{-2Nt}N^{2}tdt\|P_{n}f\|_{2}^{2}
=\frac{1}{4}\|f\|_{2}^{2}.$$
This proves the $L_{2}$ equivalence.
\end{proof}

The following estimates of the associated kernel $\partial_{t}k_{t,m}(x,y)$ can be found in \cite{T5} without verification. For the sake of completeness, we will give a sketch of the proof in the Appendix.
\begin{lem}\label{10225}\rm
	There exist two positive constants $C$ and $a$ independent of $x,y$ and $t$ such that
	\begin{enumerate}
		\item [(i)] $|\partial_{t}k_{t}(x,y)|\leq Ct^{-\frac{d}{2}-1}e^{-\frac{a}{t}|x-y|^{2}};$
		
		\item [(ii)] $|\partial_{y_{j}}\partial_{t}k_{t}(x,y)|\leq Ct^{-\frac{d}{2}-\frac{3}{2}}e^{-\frac{a}{t}|x-y|^{2}};$
		
		\item [(iii)] $|\partial_{x_{j}}\partial_{t}k_{t}(x,y)|\leq Ct^{-\frac{d}{2}-\frac{3}{2}}e^{-\frac{a}{t}|x-y|^{2}}.$
		
	\end{enumerate}
	Consequently, we have
	\begin{itemize}
		\item[(iv)] $\|\partial_{(\cdot)}k_{(\cdot)}(x,y)\|_{R_{d}}\leq \frac{C}{|x-y|^{d}};$
		
		\item[(v)] $\|\partial_{y_{j}}\partial_{(\cdot)}k_{(\cdot)}(x,y)\|_{R_{d}}\leq \frac{C}{|x-y|^{d+1}};$
		
		\item[(vi)] $\|\partial_{x_{j}}\partial_{(\cdot)}k_{(\cdot)}(x,y)\|_{R_{d}}\leq \frac{C}{|x-y|^{d+1}}.$
		
	\end{itemize}
\end{lem}

Now we are ready to prove Theorem \ref{main}.
\begin{proof}[Proof of Theorem \ref{main}]
By Lemma \ref{1052}, Lemma \ref{10225} and noting $R_d$ is separable, conclusion (i) and (ii) are the consequence of standard Hilbert-valued CZ theory (see \cite{MP,C}).

In the following, we prove (iii). Since $R_{d}$ is separable, we may find an orthonormal basis, denote by $(u_{m})_{m\geq1}$. Let
$\partial_{t}k_{t,m}(x,y)=\langle u_{m},\partial_{t}k_{t}(x,y)\rangle,$ where $\langle\ ,\ \rangle$ is the inner product induced by $R_{d}$. Denote by $\partial_{t}H^{t}_{m}$ the CZ operator associated with the kernel $\partial_{t}k_{t,m}(x,y)$.
Then conclusion (i) implies	for
$f\in L_{1}(\mathcal N)$,
$$\inf_{\partial_{(\cdot)}H_m^{(\cdot)}f=g_{m}+h_{m}}\big\{\|(g_{m})\|_{L_{1,\infty}(\mathcal{N}; R_{d}^{c})}+\|(h_{m})\|_{L_{1,\infty}(\mathcal{N}; R_{d}^{r})}\big\}\lesssim\|f\|_{1}.$$	

If we set
$$Tf(x)=\sum_{m}\varepsilon_{m}\partial_{(\cdot)}H_m^{(\cdot)}f(x),$$ where $(\varepsilon_{m})$ is the Rademacher sequence on a probability space $(\Omega,P)$, then
by noncommutative  Khintchine inequality in weak $L_{1}$-space \cite[Corollary 3.2]{C1}, we  obtain 
$$\|Tf\|_{L_{1,\infty}(L_{\infty}(\Omega)\overline{\otimes}\mathcal N)}\lesssim\|f\|_{1}.$$
Now using Lemma \ref{1052} and real interpolation \cite{P2}, we conclude that $T$ is bounded from $L_{p}(\mathcal{N})$ to $L_p(L_{\infty}(\Omega)\overline{\otimes}\mathcal{N})$ for $1<p<2$. Hence, $\partial_{(\cdot)}H^{(\cdot)}$ is bounded from $L_{p}(\mathcal{N})$ to $L_p(\mathcal{N};
R_{d}^{rc})$ for $1<p<2$.
according to the noncommutative  Khintchine inequality in $L_p$-spaces \cite{LG}.

On the other hand, if we set $T_{c}f=\sum_{m=1}^\infty \partial_{(\cdot)}H^{(\cdot)}_{m} \hskip-1pt f \otimes e_{m1}$ and $T_{r}f=\sum_{m=1}^\infty \partial_{(\cdot)}H^{(\cdot)}_{m} \hskip-1pt f \otimes e_{1m}$, then by Lemma \ref{1052} and conclusion (ii), 
$T_{c}$ and $T_{r}$ are bounded from $L_p(\mathcal N)$ to $L_p(\mathcal{N}\bar\otimes\mathcal B({\ell_{2}}))$ thanks to Masut's interpolation \cite{Mu} for $2\leq p<\infty$. Therefore, $\partial_{(\cdot)}H^{(\cdot)}$ is bounded from $L_p(\mathcal N)$ to $L_p(\mathcal{N};
R_{d}^{rc})$ for all $1< p<\infty$.

Finally, the inverse inequality can be obtained by a routine approach. Indeed, by applying the polarization identity, Lemma \ref{1052} and H\"{o}lder's inequality, we arrive at
$$\|f\|_{p}=\sup_{\|g\|_{p^{\prime}}\leq1}\langle \partial_{(\cdot)}H^{(\cdot)}f,\partial_{(\cdot)}H^{(\cdot)}g\rangle\lesssim\|\partial_{(\cdot)}H^{(\cdot)}f\|_{L_{p}(\mathcal N;R_{d}^{rc})},$$
which finishes the proof.
\end{proof}


Before proving Theorem \ref{555}, we have to introduce some more auxiliary functions. The $g_{k}$ functions are defined by $g^{c}_{1}=g^{c}$ and for $k>1$, 
$$g_{k}^{c}(f)(x)=\Big(\int_{0}^{\infty}t^{2k-1}|
\partial^{k}_{t}H^{t}f(x)|^{2}dt\Big)^{\frac{1}{2}}.$$

By applying (\ref{CS}) and the argument in presented in \cite[Page 7-8]{T5}, we can conclude that for any $k\geq1$,
\begin{equation}\label{1020}
	g^{c}_{k}(f)(x)\leq C_{k} g^{c}_{k+1}(f)(x).
\end{equation}

Another family of functions we need are the $g^{c}_{*,k}$ ($k\geq1$) functions defined by
$$g_{*,k}^{c}(f)(x)=\Big(\int_{\R^{d}}\int_{0}^{\infty}t^{\frac{2-d}{2}}
(1+t^{-1}|x-y|^{2})^{-k}|\partial_{t}H^{t}f(y)|^{2}dtdy\Big)^{\frac{1}{2}}.$$
\begin{lem}\label{t11}\rm
For $f\in L_{p}(\mathcal{N})$ with $2<p<\infty$ and $k>\frac{d}{2}$, we have
$$\|g_{*,k}^{c}(f)\|_{p}\lesssim\|f\|_{p}.$$
\end{lem}
\begin{proof}
Note that $\|g_{*,k}^{c}(f)\|^{2}_{p}=\|(g_{*,k}^{c}(f))^{2}\|_{\frac{p}{2}}$. Denote $r$ the conjugate number of $\frac{p}{2}$, and choose a positive function $h\in L_{r}(\mathcal N)$ with norm one such that
\begin{align*}
 \|g_{*,k}^{c}(f)\|^{2}_{p}
 &=\tau\int_{\R^{d}}g_{*}^{c}(f)(x)^{2}h(x)dx\\
 &= \tau\int_{\R^{d}}\int_{\R^{d}}\int_{0}^{\infty}t^{\frac{2-d}{2}}
(1+t^{-1}|x-y|^{2})^{-k}|\partial_{t}H^{t}f(y)|^{2}h(x)dtdydx\\
&=\tau\int_{\R^{d}}\int_{0}^{\infty}t|\partial_{t}H^{t}f(y)|^{2}\int_{\R^{d}}t^{-\frac{d}{2}}
(1+t^{-1}|x-y|^{2})^{-k}h(x)dxdtdy.
\end{align*}

If we set $\psi(x)=(1+|x|^2)^{-k}$, then clearly $\psi$ is positive and satisfies the conditions stated in Lemma \ref{result of cxy} when $k>\frac{d}{2}$. Hence, Lemma \ref{result of cxy} (ii) gives $(\psi_{t^{1/2}}\ast h)_{t>0}$ is of strong type $(r,r)$, which implies, by 
Remark \ref{rk:MaxFunct}, that there exists a positive operator $a\in L_{r}(\mathcal N)$ such that
$$\int_{\R^{d}}t^{-\frac{d}{2}}
(1+t^{-1}|x-y|^{2})^{-k}h(x)dx\leq a(y),\ \forall\ x\in\R^{d},\ \forall\ t>0\ \ \mbox{and}\ \ \|a\|_{r}\lesssim\|h\|_{r}=1.$$

Finally, by above relation, H\"{o}lder's inequality and Theorem \ref{main}, we get
$$\|g_{*,k}^{c}(f)\|^{2}_{p}\lesssim\tau\int_{\R^{d}}\int_{0}^{\infty}
t|\partial_{t}H^{t}f(y)|^{2}dta(y)dy\lesssim\|f\|^{2}_{p},$$
which finishes the proof.
\end{proof}
\subsection{Proof of Theorem \ref{555}} 
In this subsection, we prove Theorem \ref{555}. 
\begin{proof}[Proof of Theorem \ref{555}]
Let $F(x)=T_{\mu}f(x)$. We claim it is enough to show for all $x\in\R^{d}$ and all $k>d/2$
\begin{equation}\label{7011}
	g^{c}_{k+1}(F)(x)\leq C_{k,\mu}g_{*,k}^{c}(f)(x).
\end{equation} 
Indeed, once we obtain (\ref{7011}), Theorem \ref{555} for $p>2$ follows from Theorem \ref{main} (iii) and Lemma \ref{t11} in view of (\ref{1020}): $g^{c}(F)(x)\leq C_{k}g_{k+1}^{c}(F)(x)$; and Theorem \ref{555} for $1<p<2$ can be obtained by duality. So we get the desired estimate.

To prove (\ref{7011}), we start by defining the following function
$$M(t,x,y)=\sum_{\nu} e^{-(2|\nu|+d)t}\mu(2|\nu|+d)\Phi_{\nu}(x)\Phi_{\nu}(y).$$
If we set $u(x,t)=H^{t}f(x)$ and $U(x,t)=H^{t}F(x)$, then it is easy to verify that
\begin{equation}\label{701221}
	U(x,t_{1}+t_{2})=\int_{\R^{d}}M(t_{1},x,y)u(y,t_{2})dy.
\end{equation}
Differentiating (\ref{701221}) $k$ times with respect to $t_{1}$ and one time with respect to $t_{2}$ and then setting $t_{1}=t_{2}=t/2$, we get 
\begin{equation}\label{7012212}
	\partial_{t}^{k+1}U(x,t)=\int_{\R^{d}}\partial_{t}^{k}M(t/2,x,y)\partial_{t}u(y,t/2)dy.
\end{equation}

According to (\ref{7012212}), we write $\partial_{t}^{k+1}U(x,t)=A_{t}(x)+B_{t}(x)$ with
\begin{equation*}\label{7012}
	A_{t}(x)= \int_{|x-y|\leq t^{1/2}}\partial_{t}^{k}M(t/2,x,y)\partial_{t}u(y,t/2)dy,
\end{equation*}
and
\begin{equation*}\label{7013}
	B_{t}(x)= \int_{|x-y|> t^{1/2}}\partial_{t}^{k}M(t/2,x,y)\partial_{t}u(y,t/2)dy.
\end{equation*}
Now we use the operator convexity inequality of square function $x\mapsto|x|^{2}$ to obatin
$$|\partial_{t}^{k+1}U(x,t)|^{2}\leq2(|A_{t}(x)|^{2}+|B_{t}(x)|^{2}).$$
The following estimate has already been shown in \cite{T4} that
\begin{equation*}
	|A_{t}(x)|^{2}+|B_{t}(x)|^{2}\lesssim t^{-d/2-2k}\int_{\R^{d}}(1+t^{-1}|x-y|^{2})^{-k}|\partial_{t}u(y,t)|^{2}dy.
\end{equation*}
Therefore, we finally deduce that
$$\int_{0}^{\infty}|\partial_{t}^{k+1}U(x,t)|^{2}t^{2k+1}dt\lesssim \int_{\R^{d}}\int_{0}^{\infty}t^{-d/2+1}(1+t^{-1}|x-y|^{2})^{-k}
|\partial_{t}u(y,t)|^{2}dtdy,$$
which establishes (\ref{7011}) via the operator monotonicity  of $x\rightarrow x^{t}$, $x\geq0$ for $0<t<1$.
\end{proof}

\section{Oscillation operator related to Hermite expansion}
In this section, we are going to investigate another multiplier theorem, where the function $\mu$ is defined as 
\begin{equation}\label{1008}
	\mu(n)=(2n+d)^{-\alpha}e^{(2n+d)it}.
\end{equation}

For $f\in C_{c}^{\infty}(\R^d)\otimes \mathcal{S}_{\M}$, we define $T_{t}^{\alpha}$ as
\begin{equation}\label{1009}
T_{t}^{\alpha}f(x)=\sum_{n=0}^{\infty}(2n+d)^{-\alpha}e^{(2n+d)it}
P_{n}f(x).
\end{equation}

Observe that these operators behave like the operator given by convolution associated with the oscillating kernels $|x|^{-\alpha}e^{i\langle x,t\rangle}$. A simple calculation shows that this function $\mu$ defined as (\ref{1008}) does not satisfy the conditions in Theorem \ref{555} unless $\alpha>d$.

The following theorem is our main result in this section.
\begin{thm}\label{C3}\rm
Let $1\leq p<\infty$.
\begin{enumerate}
	\item [(i)] For $p=1$, $\alpha=d/2$ and $t\in [t_{0},\pi/4]$, where $t_{0}$ is an arbitrary positive constant, $T_{t}^{\alpha}$ is bounded from $\mathrm{H}_1(\R^d;\M)$ to $L_{1}(\mathcal N)$. That is,
	$$\|T_{t}^{\alpha}f\|_{1}\leq C_{d,t_0}\|f\|_{\mathrm{H}_1(\R^d;\M)}.$$
	\item [(ii)] For $1<p<\infty$, $\alpha=d|1/p-1/2|$ and $t\in [t_{0},\pi/4]$, the operator $T_{t}^{\alpha}$ are bounded on $L_{p}(\mathcal N)$. That is for $f\in L_{p}(\mathcal N)$,
	$$\|T_{t}^{\alpha}f\|_{p}\leq C_{p,d,t_0}\|f\|_{p}.$$
\end{enumerate}
\end{thm}
\begin{rk}\rm
(i) In the classical case, Theorem \ref{C3} extends the Hardy-Littlewood theorem for the Fourier transform (see for instance \cite{HL,Sa}); (ii) Theorem \ref{C3} is some sharp estimate of Schr\"{o}dinger group for Hermite operator, which is a case study of a forthcoming paper by Fan, Hong and Wang \cite{FHW}.
\end{rk}

In the rest of this section, we prove Theorem \ref{C3} when $d=1$, since there is absolutely no change of the proof for the general case. For convenience, we set the operator $T_{t}^{1/2}=L_{t}$. In this case, the kernel $K_{t}(x,y)$ associated to $L_{t}$ is given by
$$K_{t}(x,y)=\sum_{n=1}^{\infty}(2n+1)^{-1/2}e^{(2n+1)it}\varphi_{n}(x)\varphi_{n}(y).$$
If we set
$$K_{t}^{*}(x,y,\lambda)=\sum_{n=1}^{\infty} e^{(2n+1)(-\lambda+it)}\varphi_{n}(x)\varphi_{n}(y),$$
then we can express the kernel $K_{t}(x,y)$ as
$$K_{t}(x,y)=c\int_{0}^{\infty}\lambda^{-1/2}
K_{t}^{*}(x,y,\lambda)d\lambda,$$
where $c=1/\Gamma(1/2)$.
By \cite{T4},
$$K_{t}^{*}(x,y,\lambda)=c(\sinh2(\lambda-it))^{-1/2}
e^{-A_{t}(x,y,\lambda)}e^{iB_{t}(x,y,\lambda)}$$
where $A_{t}(x,y,\lambda)$ and $B_{t}(x,y,\lambda)$ are given by
\begin{eqnarray*}
	2A_{t}(x,y,\lambda) & =  (\sinh^{2}2\lambda+\sin^{2}2t)^{-1}(\sinh2\lambda)\{\cos2t(x-y)^{2}\\ & +
	(\cosh2\lambda-
	\cos2t)(x^{2}+y^{2})\}.
\end{eqnarray*}
\begin{eqnarray*}
	\ \ \ \ 2B_{t}(x,y,\lambda) & = -(\sinh^{2}2\lambda+\sin^{2}2t)^{-1}(\sinh2t)\{\cosh2\lambda(x-y)^{2}\\ & -
	(\cosh2\lambda-\cos2t)(x^{2}+y^{2})\}.
\end{eqnarray*}

It is also shown in \cite{T4} that the following integral
$$\int_{1}^{\infty}\lambda^{-1/2}K_{t}^{*}(x,y,\lambda)d\lambda$$
defines a nice $L_1$ kernel and hence the operator corresponding to this kernel is bounded on $L_{p}(\mathcal{N})$ for all $1\leq p\leq\infty$. So in the following, we may consider the kernel  given by
\begin{equation}\label{mainkernel}
	K_{t}(x,y)=\int_{0}^{1}\lambda^{-1/2}K_{t}^{*}(x,y,\lambda)d\lambda.
\end{equation}

To prove Theorem \ref{C3}, we need following certain estimates of the kernel. 
For convenience, we write $K_{t}(x,y,\lambda)=\{\sinh2(\lambda-it)\}^{-1/2}e^{-A_{t}(x,y,\lambda)}$. The proof of the following lemma can be found in the Appendix.
\begin{lem}\label{2002}\rm
Let $0<t\leq \pi/4$. Then following estimates hold:
\begin{enumerate}
\item [(i)] $\big|\int_{0}^{1}\lambda^{-1/2}
K_{t}(x,y,\lambda)d\lambda\big|\leq C|x-y|^{-1}$;

\item [(ii)] $\big|\int_{0}^{1}\lambda^{-1/2}\partial_{y}K_{t}(x,y,\lambda)
e^{iB_{t}(x,y,\lambda)}d\lambda\big|\leq C|x-y|^{-2}$;

\item [(iii)] $\big|\int_{0}^{1}\lambda^{-1/2}K_{t}(x,y,\lambda)
\partial_{y}\{e^{iB_{t}(x,y,\lambda)}\}d\lambda\big|\leq C(\sin2t)^{-3/2}$;

\item [(iv)] $\big|\int_{0}^{1}\lambda^{-1/2}\lambda K_{t}(x,y,\lambda)
\partial_{\lambda}\{e^{iB_{t}(x,y,\lambda)}\}d\lambda\big|\leq C|x-y|^{-3}$.
\end{enumerate}
\end{lem}

We are now ready to prove Theorem \ref{C3}. The idea comes from \cite[Proposition 4.1]{T4}.
\begin{proof}[Proof of Theorem \ref{C3}]
(i) We first prove the $(\mathrm{H}_{1},L_{1})$ estimate of $L_{t}$. 
The proof is based on the atomic decomposition of $\mathrm{H}_1^c(\R;\M)$ introduced in Section 2.
Moreover, it suffices to show that for any atom $a$
$$\big\|L_{t}a\big\|_{1}
\lesssim1.$$

In what follows, we always assume that the atom $a$ is supported in $Q_{\delta}$, where $Q_{\delta}$ denote the interval of length $2\delta$ with center $c_{Q}$. Let $Q_{2\delta}$ be the interval with center $c_{Q}$ with length $4\delta$. Denote by $Q^{c}_{\delta}$ the complement of $Q_{\delta}$. Set $F(x)=T_{t}a(x)$. Decompose $F$ as a sum of three functions $F=F_{1}+F_{2}+F_{3}$, where $F_{1}=F\chi_{Q_{2\delta}}$, $F_{2}=F\chi_{Q_{2\delta}^{c}\cap Q_{\delta^{-1}}}$ and $F_{3}=F\chi_{Q_{2\delta}^{c}\cap Q_{\delta^{-1}}^{c}}$. 
Then by Minkowski's inequality, we have
$$\big\|L_{t}a\big\|_{1}
\leq\|F_{1}\|_{1}+\|F_{2}\|_{1}+\|F_{3}\|_{1}.$$

We first estimate $F_{1}$. 
It suffices to show that $L_{t}$ bounded on $L_1(\M; L_2^c(\R))$. Indeed, by this conclusion and (\ref{CS}),
we deduce that
\begin{eqnarray*}
\big\| F_{1} \big\|_{1} & = &  \int_{Q_{2\delta}} \tau (|F(x)|) \, dx   \le |Q_{2\delta}|^{1/2} \, \tau \Big[ \Big( \int_{Q_{2\delta}} |L_{t}a(x)|^{2} \, dx \Big)^{1/2} \Big] \\ & \lesssim & |Q_{2\delta}|^{1/2} \, \tau \Big[ \Big( \int_{Q_{\delta}} |a(x)|^2 \, dx \Big)^{1/2} \Big] \, \lesssim \, 1,
\end{eqnarray*}
where the last inequality follows the properties of atoms stated in subsection 2.5. This gives the desired estimate. 

Now we examine the $L_1(\M; L_2^c(\R))$-boundedness of $L_{t}$. Let $f\in L_1(\M; L_2^c(\R))$. Then  by anti-linear duality, 
$$\|L_{t}f\|_{L_1(\M; L_2^c(\R))}\leq\sup_{\|h\|_{L_{\infty}(L_2^c)}\leq1}\|L^{*}_{t}h\|_{L_\infty(\M; L_2^c(\R))}\|f\|_{L_1(\M; L_2^c(\R))}.$$
Note that $L_{t}$ is bounded on $L_{2}(\mathcal{N})$, so is $L^{*}_{t}$. This gives 
\begin{align*}
\|L^{*}_{t}h\|_{L_\infty(\M; L_2^c(\R))}
 &=\Big\|\Big(\int_{\R}|L^{*}_{t}h(x)|^{2}dx\Big)^{1/2}\Big\|_{\M}\\
 &=\sup_{\|u\|_{L_{2}(\M)\leq1}} \Big(\int_{\R}\langle|L^{*}_{t}h(x)|^{2}u,u\rangle_{L_{2}(\M)}dx\Big)^{1/2}\\
 &=\sup_{\|u\|_{L_{2}(\M)\leq1}} \Big(\int_{\R}\|L^{*}_{t}(hu)(x)\|^{2}_{L_{2}(\M)}dx\Big)^{1/2}\\
 &\lesssim\sup_{\|u\|_{L_{2}(\M)\leq1}} \Big(\int_{\R}\|h(x)u\|^{2}_{L_{2}(\M)}dx\Big)^{1/2}\\
 &=\Big\|\Big(\int_{\R}|h(x)|^{2}dx\Big)^{1/2}\Big\|_{\M},
\end{align*}
which finishes the argument.

The second term $F_2$ can be verified by using the argument presented in \cite[Proposition 4.1]{T4},  Lemma \ref{2002} and (\ref{CS}). So we omit the proof.

We then turn to the last term $F_{3}$. Decompose the kernel $K_{t}(x,y)$ as follows
$$K_{t}(x,y)=E_{t}(x,y)+G_{t}(x,y)+J_{t}(x,y),$$
where these three terms above are defined by
\begin{eqnarray*}
	E_{t}(x,y) & = &\int_{0}^{1}\lambda^{-1/2}\{K_{t}(x,y,\lambda)
	-K_{t}(x,c_{Q},\lambda)\}e^{iB_{t}(x,y,\lambda)}d\lambda,\\
	G_{t}(x,y) & = & \int_{0}^{1}\lambda^{-1/2}K_{t}(x,c_{Q},\lambda)
	\{e^{iB_{t}(x,y,\lambda)}-e^{iB_{t}(x,y,0)}\}d\lambda, \\
	J_{t}(x,y) & = & \int_{0}^{1}\lambda^{-1/2}K_{t}(x,c_{Q},\lambda)
	e^{iB_{t}(x,y,0)}d\lambda.
\end{eqnarray*}

Denote by the operators $E_{t}$ associated to $E_{t}(x,y)$,  $G_{t}$ associated to the kernel $G_{t}(x,y)$ and $J_{t}$ associated to the kernel $J_{t}(x,y)$. Then Minkowski's inequality implies that
$$\|F_{3}\|_{1}\leq\|E_{t}a\chi_{Q_{2\delta}^{c}\cap Q_{\delta^{-1}}^{c}}\|_{1}+\|G_{t}a\chi_{Q_{2\delta}^{c}\cap Q_{\delta^{-1}}^{c}}\|_{1}+\|J_{t}a\chi_{Q_{2\delta}^{c}\cap Q_{\delta^{-1}}^{c}}\|_{1}.$$
Hence, it is sufficient to show
\begin{eqnarray}\label{500014}
	\max\big\{\|E_{t}a\chi_{Q_{2\delta}^{c}\cap Q_{\delta^{-1}}^{c}}\|_{1},\|G_{t}a\chi_{Q_{2\delta}^{c}\cap Q_{\delta^{-1}}^{c}}\|_{1},\|J_{t}a\chi_{Q_{2\delta}^{c}\cap Q_{\delta^{-1}}^{c}}\|_{1}\big\}\lesssim1.
\end{eqnarray}

In the following, we just estimate the term that  involves $J_t$, since the other two terms can be done 
by using the argument in  \cite[Proposition 4.1]{T4}.
 
By Fubini's theorem, we rewrite $J_{t}a(x)$ as
$$J_{t}a(x)=F_{t}a(x)\int_{0}^{1}\lambda^{-1/2}
K_{t}(x,c_{Q},\lambda)d\lambda\triangleq F_{t}a(x)g_{t}(x),$$
where the operator $F_{t}$ is defined by
$$F_{t}a(x)=\int_{\R}e^{iB_{t}(x,y,0)}a(y)dy.$$
By Lemma \ref{2002} (i), we know that $|g_{t}(x)|\lesssim|x-c_{Q}|^{-1}$. On the other hand, by the definition of $B_{t}(x,y,0)$, we obtain
$$F_{t}a(x)=p_t(x)\int_{\R}e^{ic_txy}a(y)p_t(y)dy,$$
where $p_t(x)=e^{-i\frac{c_tx^2\cos2t}{2}}$ and $c_t=(\sin^{2}2t)^{-1}(\sinh2t)$. Since $t_0<t<\pi/4$, we have $|c_t|\approx1$ with the constants depend on $t_0$. Therefore, noting $|p_t(\cdot)|=1$ and using 
Plancherel's theorem, we get $F_{t}$ is bounded on $L_{2}(\mathcal{N})$. As a consequence, $F_{t}$ is bounded on $L_1(\M; L_2^c(\R))$ by the same argument as in estimating $F_{1}$. Finally, using (\ref{CS}), we see that
\begin{eqnarray*}
	\big\| J_{t}a\chi_{Q_{2\delta}^{c}\cap Q_{\delta^{-1}}^{c}} \big\|_{1} & \leq &
	\Big(\int_{|x-c_{Q}|\geq\delta^{-1}}|g_{t}(x)|^{2}dx\Big)^{1/2}\tau \Big[ \Big( \int_{Q_{\delta}} |F_{t}a(x)|^2 \, dx \Big)^{1/2} \Big] \\
	& \lesssim &   \Big(\int_{|x-c_{Q}|\geq\delta^{-1}}|x-c_{Q}|^{-2}dx\Big)^{1/2}
	\tau \Big[ \Big( \int_{Q_{\delta}} |a(x)|^2 \, dx \Big)^{1/2} \Big]
	\lesssim1,
\end{eqnarray*}
where the first inequality follows the fact that if $x\in Q_{2\delta}^{c}\cap Q_{\delta^{-1}}^{c}$, then $|x-c_{Q}|\geq\delta^{-1}$.

Combining the estimates obtained so far, we finish the proof of (\ref{500014}), which actually gives the desired $(\mathrm{H}_1,L_{1})$ boundedness of $L_{t}$.

(ii) The strong type $(p,p)$ $(1<p<2)$ estimate of $T_{t}^{\alpha}$ follows from Fefferman-Stein's interpolation theorem \cite{FS} by considering the sequence of operator $S_{z}=T_{t}^{(1-z)/2}$; while 
the strong type $(p,p)$ ($2<p<\infty$) can be obtained by duality. Hence, we finish the proof.
\end{proof}

\section*{Appendix. Proof of Lemma \ref{10225} and Lemma \ref{2002} }
In this appendix, we present the proof of Lemma \ref{10225} and Lemma \ref{2002}.
\begin{proof}[Proof of Lemma \ref{10225}]
	(i) 
	Since the kernel $k_{t}(x,y)$ is the product of one dimensional kernels $k_{t}(x_{j},y_{j})$, it suffices to consider $d=1$. In this case,
	$$k_{t}(x,y)=(\sinh 2t)^{-\frac{1}{2}}e^{-\varphi(t,x,y)},$$
	where $\varphi(t,x,y)=\frac{1}{2}(x-y)^{2}\coth2t+xy\tanh t$.
	It is not difficult to verify that $\frac{1}{4}(x-y)^{2}\coth2t+xy\tanh t$ is nonnegative. Hence we get
	\begin{equation}\label{1080}
		e^{-\varphi(t,x,y)}\leq e^{-\frac{1}{4}(x-y)^{2}\coth2t}.
	\end{equation}

	A simple calculation shows that
	\begin{align*}
		\partial_{t}k_{t}(x,y)
		=&\ -[(\sinh2t)^{-\frac{3}{2}}\cosh2te^{-\varphi(t,x,y)}
		+(\sinh2t)^{-\frac{1}{2}}e^{-\varphi(t,x,y)}\partial_{t}\varphi(t,x,y)]\\
		=:&\ -[A+B],
	\end{align*}
	where
	\begin{equation*}
		\partial_{t}\varphi(t,x,y)=-\frac{(x-y)^{2}}{(\sinh2t)^{2}}+\frac{xy}{(\cosh t)^{2}}.
	\end{equation*}

	Consider $0<t<1$ firstly. Note that  $\sinh t$ behave like $t$ and $\cosh 2t=O(1)$. Then combining with (\ref{1080}), we get
	\begin{equation}\label{1083}
		|A|\lesssim t^{-\frac{3}{2}} e^{-\frac{1}{8}\frac{(x-y)^{2}}{t}}.
	\end{equation}

	To estimate $B$. We decompose $B$ as
	\begin{align*}
		B=&\ -\Big[(\sinh2t)^{-\frac{1}{2}}\frac{(x-y)^{2}}{(\sinh2t)^{2}}e^{-\varphi(t,x,y)}-
		(\sinh2t)^{-\frac{1}{2}}\frac{xy}{(\cosh t)^{2}}e^{-\varphi(t,x,y)}\Big]\\
	=	:&\ -[B_{1}-B_{2}].
	\end{align*}
For term $B_1$, we have
	\begin{equation}\label{10081}
		|B_{1}|\lesssim t^{-\frac{5}{2}}(x-y)^{2}e^{-\frac{1}{8}\frac{(x-y)^{2}}{t}}\lesssim
		t^{-\frac{3}{2}}e^{-\frac{1}{16}\frac{(x-y)^{2}}{t}}.
	\end{equation}

	Now we deal with $B_{2}$. Note that when $xy\geq0$, it is easy to see $xye^{-xy\tanh t}$ is bounded by a constant times $t^{-1}$ and hence
	\begin{equation}\label{1084}
		|B_{2}|\lesssim t^{-\frac{1}{2}}t^{-1} e^{-\frac{1}{16}\frac{(x-y)^{2}}{t}}=t^{-\frac{3}{2}}e^{-\frac{1}{16}\frac{(x-y)^{2}}{t}}.
	\end{equation}
	On the other hand, when $xy<0$, then $|xy|=-xy\leq (x-y)^{2}$ and whence
	\begin{equation}\label{1085}
		|B_{2}|\lesssim t^{-\frac{1}{2}}(x-y)^{2} e^{-\frac{1}{8}\frac{(x-y)^{2}}{t}}\lesssim t^{-\frac{3}{2}}e^{-\frac{1}{16}\frac{(x-y)^{2}}{t}}.
	\end{equation}
	Therefore, by (\ref{10081}), (\ref{1084}) and (\ref{1085}), we find
	\begin{equation}\label{1086}
		|B|\lesssim t^{-\frac{3}{2}}e^{-\frac{1}{16}\frac{(x-y)^{2}}{t}}.
	\end{equation}

	Finally, (\ref{1083}) and (\ref{1086}) give
	\begin{equation*}
		|\partial_{t}k_{t}(x,y)|\lesssim t^{-\frac{3}{2}}e^{-\frac{1}{16}\frac{(x-y)^{2}}{t}}.
	\end{equation*}
	Thus we get the desired estimate when $0<t<1$. The case of $t\geq1$ is quite similar as previous, just noting that for $t\geq1$ both $\sinh2t$ and $\cosh2t$ behave like $e^{2t}$, while this condition is stronger than before. The details are omitted. 
	
	(ii) Let us verify the smoothness condition. Note that $\Phi$ can be rewritten as
	$$\Phi_t(x,y)=-\frac{(x-y)^{2}}{2\sinh 2t}-\frac{(x^{2}+y^{2})\tanh t}{2}.$$
	
	By symmetry, we just need to estimate $\partial_{y}\partial_{t}k_{t}(x,y)$. However, $\partial_{y}\partial_{t}k_{t}(x,y)$ has many terms we indicate how to estimate one typical term:
	$$J=(x-y)^{3}(\sinh 2t)^{-\frac{7}{2}}\cosh 2te^{\Phi(t)}.$$
	
	 For $0<t<1$, using the similar argument as in (i), we deduce
	$$|J|\lesssim t^{-\frac{7}{2}}|x-y|^{3}e^{-\frac{|x-y|^{2}}{8t}}\lesssim t^{-2}e^{-\frac{|x-y|^{2}}{16t}}.$$
	The other terms can be proved in a similar fashion. The case $t\geq1$ can also be obtained without much difficulty. The details are omitted. This completes the proof.	
\end{proof}

\begin{proof}[Proof of Lemma \ref{2002}]
	We only verify conclusion (i). The other three estimates we refer to \cite[Lemma 4.1]{T4}. Consider $0<t\leq\pi/8$ firstly. We claim that
	\begin{equation}\label{10833}
	|K_{t}(x,y,\lambda)|\lesssim\lambda^{-1/2}e^{-c\lambda^{-1}(x-y)^{2}}\leq\lambda^{-1}e^{-c\lambda^{-1}(x-y)^{2}}.
	\end{equation}
Indeed, a simple calculation shows
	$$|\sinh2(\lambda-it)|^{2}=c(\sinh^{2}2\lambda+\sin^{2}2t).$$
	Since $0<t\leq\pi/8$ and $0\leq\lambda\leq1$, $\cos2t\geq2^{-1/2}$, $\sinh2\lambda$ behaves like $\lambda$ and $\cosh2\lambda=O(1)$.  
	Assume further $t\leq\lambda$. Then  $(\sinh^{2}2\lambda+\sin^{2}2t)$ behaves like $\lambda^{2}$. Combining with the observation $\cosh2\lambda-\cos2t\geq0$, we deduce that 
	$$|K_{t}(x,y,\lambda)|\lesssim\lambda^{-1/2}e^{-c\lambda^{-1}(x-y)^{2}}\leq\lambda^{-1}e^{-c\lambda^{-1}(x-y)^{2}}.$$
	Integrating (\ref{10833}) against $\lambda^{-1/2}$, we have
	\begin{eqnarray*}
		\big|\int_{0}^{1}\lambda^{-1/2}
		K_{t}(x,y,\lambda)d\lambda\big|
		& \lesssim & \int_{0}^{1}\lambda^{-3/2}e^{-c\lambda^{-1}(x-y)^{2}}d\lambda\\
		& = &\int_{1}^{\infty}\lambda^{-1/2}e^{-c\lambda(x-y)^{2}}d\lambda\\
		& \leq &\int_{0}^{\infty}\lambda^{-1/2}
		e^{-c\lambda(x-y)^{2}}d\lambda\lesssim|x-y|^{-1}.
	\end{eqnarray*}

	Under the assumption $t>\lambda$,  $(\sinh^{2}2\lambda+\sin^{2}2t)$ behaves like $t^{2}$. Hence,
	$$|K_{t}(x,y,\lambda)|\lesssim t^{-1}e^{-c\lambda t^{-2}(x-y)^{2}}.$$
	Then integrating this against $\lambda^{-1/2}$, we get
	\begin{eqnarray*}
		\big|\int_{0}^{1}\lambda^{-1/2}
		K_{t}(x,y,\lambda)d\lambda\big|
		& \lesssim & t^{-1}\int_{0}^{1}\lambda^{-1/2}e^{-c\lambda t^{-2}(x-y)^{2}}d\lambda\\
		& \leq &t^{-1}\int_{0}^{\infty}\lambda^{-1/2}e^{-c\lambda t^{-2}(x-y)^{2}}d\lambda \lesssim |x-y|^{-1}.
	\end{eqnarray*}
	This proves (i) when $0<t\leq \pi/8$.
	
	We now turn to estimate (i) in the case of $\pi/8<t\leq\pi/4$. Observe that if $t$ is in the neighbourhood of $\pi/4$ we can use
	$$(\sinh^{2}2\lambda+\sin^{2}2t)^{-1}(\sinh2\lambda)(\cosh2\lambda-
	\cos2t)(x^{2}+y^{2})$$
	in place of
	$$(\sinh^{2}2\lambda+\sin^{2}2t)^{-1}(\sinh2\lambda)
	\cos2t(x-y)^{2}$$
	since $(\cosh2\lambda-\cos2t)\geq1-\cos2t\geq C$. The following argument is the same as the case of $0<t\leq\pi/8$. The completes the proof.
\end{proof}

\noindent {\bf Acknowledgements} \

\medskip

I am very grateful to my supervisor Guixiang Hong for many valuable suggestions on this topic,
his guidance throughout the making of this paper, Dr. Liang Wang for the useful discussion.  I would like to thank the referees for their very
careful reading and valuable comments.

\end{document}